\documentclass[11pt, a4paper]{amsart}
\usepackage[utf8]{inputenc}
\usepackage{cite}
\usepackage[top=3cm, bottom=3cm, left=3.5cm, right=3.5cm] {geometry}
\usepackage[T1]{fontenc}
\usepackage{amsmath}
\usepackage{amsthm}
\usepackage{amsfonts}
\usepackage{adjustbox}
\usepackage[english]{babel}
\usepackage{array,epsfig}
\usepackage{amssymb}
\usepackage{amsxtra}
\usepackage{amsthm}
\usepackage{latexsym}
\usepackage{dsfont}
\usepackage{mathrsfs}
\usepackage[mathscr]{eucal}
\usepackage{color}
\usepackage[all]{xy}
\usepackage{hyperref}
\usepackage{graphicx}
\usepackage{mathtools}
\usepackage{caption}
\usepackage{relsize}

\usepackage{dynkin-diagrams}
\usepackage{tikz}

\newtheorem{thm}{Theorem}[section]
\newtheorem{thmx}{Theorem}

\newtheorem{corx}[thmx]{Corollary}
\newtheorem{lem}[thm]{Lemma}
\newtheorem{prop}[thm]{Proposition}

\newtheorem{defn}[thm]{Definition}
\newtheorem{rmk}[thm]{Remark}
\newtheorem{ex}[thm]{Example}

\newcommand{\set}[1]{\left\{#1\right\}}
\newcommand{\tuple}[1]{\left(#1\right)}

\newcommand{\abs}[1]{\left|#1\right|}
\newcommand{\norm}[1]{\left\|#1\right\|}
\newcommand{\sprod}[1]{\left<#1\right>}

\newcommand{\ol}[1]{\overline{#1}}
\newcommand{\wh}[1]{\widehat{#1}}
\newcommand{\wt}[1]{\widetilde{#1}}

\newcommand{\afrak}{\mathfrak{a}}

\newcommand{\kfrak}{\mathfrak{k}}

\newcommand{\mfrak}{\mathfrak{m}}

\newcommand{\tfrak}{\mathfrak{t}}

\newcommand{\Cbb}{\mathbb{C}}

\newcommand{\Nbb}{\mathbb{N}}

\newcommand{\Pbb}{\mathbb{P}}
\newcommand{\Qbb}{\mathbb{Q}}
\newcommand{\Rbb}{\mathbb{R}}
\newcommand{\Sbb}{\mathbb{S}}

\newcommand{\Zbb}{\mathbb{Z}}

\newcommand{\Ccal}{\mathcal{C}}
\newcommand{\Dcal}{\mathcal{D}}

\newcommand{\Fcal}{\mathcal{F}}

\newcommand{\Ical}{\mathcal{I}}

\newcommand{\Kcal}{\mathcal{K}}
\newcommand{\Lcal}{\mathcal{L}}
\newcommand{\Mcal}{\mathcal{M}}

\newcommand{\Ocal}{\mathcal{O}}

\newcommand{\Rcal}{\mathcal{R}}

\newcommand{\Vcal}{\mathcal{V}}
\newcommand{\Xcal}{\mathcal{X}}
\newcommand{\Ycal}{\mathcal{Y}}
\newcommand{\Ncal}{\mathcal{N}}

\renewcommand{\phi}{\varphi}

\newcommand{\del}{\partial}
\newcommand{\delb}{\overline{\partial}}

\newcommand{\Aut}{\operatorname{Aut}}
\newcommand{\Sp}{\operatorname{Sp}}
\newcommand{\SO}{\operatorname{SO}}

\newcommand{\Fut}{\operatorname{Fut}}

\newcommand{\SL}{\operatorname{SL}}

\newcommand{\vol}{\operatorname{vol}}

\title[K-stable valuations and Calabi-Yau metrics]{K-stable valuations and Calabi-Yau metrics on affine spherical varieties}
\author{Tran-Trung Nghiem}
\address{Tran-Trung Nghiem, IMAG, Univ Montpellier, CNRS, Montpellier, France}
\email{tran-trung.nghiem@umontpellier.fr}

\begin{document}
\maketitle
\begin{abstract}
After providing an explicit K-stability condition for a \( \Qbb\)-Gorenstein log spherical cone, we prove the existence and uniqueness of an equivariant K-stable degeneration of the cone, and deduce uniqueness of the asymptotic cone of a given complete \( K \)-invariant Calabi-Yau metric in the trivial class of an affine \( G\)-spherical manifold, \( K \) being the maximal compact subgroup of \( G \). 

Next, we prove that the valuation induced by \( K \)-invariant Calabi-Yau metrics on affine \( G\)-spherical manifolds is in fact \( G\)-invariant. As an application, we point out an affine smoothing of a Calabi-Yau cone that does not admit any \( K\)-invariant Calabi-Yau metrics asymptotic to the cone. Another corollary is that on $\Cbb^3$, there are no other complete Calabi-Yau metrics with maximal volume growth and spherical symmetry other than the standard flat metric and the Li-Conlon-Rochon-Székelyhidi metrics with horospherical asymptotic cone. This answers the question whether there is a nontrivial asymptotic cone with smooth cross section on \( \Cbb^{3} \) raised by Conlon-Rochon when the symmetry is spherical. 
\end{abstract}

\section{Introduction}

\subsection{Background}
The Yau-Tian-Donaldson correspondence establishes an equivalence between the existence of canonical metrics and an algebro-geometric stability condition. Large progress has recently been made for Ricci-flat Kähler cone metrics (also called conical Calabi-Yau metrics) on a \textit{Fano cone}, which is basically an affine cone with respect to a polarization over a log Fano base, hence comes with an effective complex torus action. 

In base-independent terms, given a complex algebraic torus \( T \), a Fano cone \( Y \) is a \( \Qbb\)-Gorenstein klt \( T \)-affine variety with an effective \( T\)-action and a unique fixed point under \( T \) \cite{LWX}. The \textit{Reeb cone} of \( Y\) consists of elements \( \xi \) in the compact Lie algebra of \( T \) acting with positive weights on non-zero elements of \( \Cbb[Y] \). 

A \textit{conical Calabi-Yau metric} on \( (Y,J_Y) \) is a \( \del_{J_Y}  \delb_{J_Y}\)-exact (weak) Ricci-flat metric \( \omega \) with potential \( r^2 \), compatible with the weak complex structure \( J_Y \), and homogeneous under the scaling vector field generated by \( r \), i.e.
\[ \Lcal_{r \del_{r}} \omega = 2 \omega. \] 
In particular, the \( \xi = - J_Y (r \del_{r}) \) is a Reeb vector generating a holomorphic isometric action of a compact torus \( T_{\xi,c} \) on \( Y \) \cite[Lemma 2.17]{DS17}. 

Fano cones offer very rich geometry as they contain contact geometric structures, as well as underlying Fano orbifold structures. They serve as asymptotic models for Calabi-Yau metrics on affine manifolds in \cite{Li19} \cite{CR21} \cite{Sze19} \cite{BD19} \cite{Ngh24}, but also as local tangent cones to Kähler-Einstein metrics \cite{HS17}. 

Through the pioneering works of \cite{CS18}, \cite{CS19}, \cite{HL23}, \cite{Li21}, it is now established that a Fano cone has a Ricci-flat Kähler cone metric if and only if it is K-stable. More precisely, when the cone has a unique singularity, K-stability of a polarized cone \( (Y,\xi) \) is shown in \cite{CS19} to be equivalent to a K-stability condition that extends the Fano orbifold stability of Ross-Thomas \cite{RT11}. 

The general \( \Qbb\)-Gorenstein case was solved by C. Li \cite{Li21} by using the equivalence between (weak) Ricci-flat Kähler cone metrics on \( (Y,\xi) \) and certain \( g\)-solitons over quasi-regular Fano orbifold quotients of \( Y \). The \( g \)-soliton equations have moreover the same form when passing from one Reeb vector to another while keeping the underlying CR structure. The K-stability of the Fano cone is then equivalent to weighted K-stability of all quasi-regular quotients of \( Y \). 

Varieties with low complexity have been known to provide concrete examples to test K-stability criteria \cite{Del20}, \cite{IS17}. The complexity of a variety with a regular action of a reductive group \( G \) is basically the codimension of a generic Borel orbit. Normal varieties with complexity zero are called \textit{spherical varieties}. Equivalently, a \( G\)-variety is spherical if and only if it has a open dense orbit under the action of a Borel subgroup of \( G \). 

A simple \( G\)-spherical affine variety \( Y \) is said to be a \textit{\( G\)-spherical cone} if its unique closed orbit is the fixed point of \( G \). In fact, a \( \Qbb\)-Gorenstein \( G\)-spherical cone is always a Fano cone with respect to the action of a torus compatible with \( G \) \cite{Ngh22b}. 

\subsection{K-stable degeneration and K-stability of spherical cones} 
Let \( Y \) be an \( n\)-dimensional \( \Qbb\)-Gorenstein conical embedding of a spherical space \( G/H \) with colored cone \( (\Ccal_Y, \Dcal_Y) \) and set of \( G\)-invariant divisors \( \Vcal_Y \) (identified with their \( G\)-invariant valuations). Let \( T_H = \Aut_G(Y)^0 \simeq (N_G(H)/H)^0  \) be the connected component of the automorphism group of \( Y \) compatible with \( G \). 

%Let us first recall the description of the canonical divisor by \cite{Bri97a} \cite{GH15a}. 
%Since \( Y \) is \( \Qbb\)-Gorenstein and simple, there exists a canonical element \( \varpi \in \Mcal_{\Qbb} := \Mcal \otimes \Qbb\) such that
%\[ - K_Y = \sum_{\nu \in \Vcal_Y } D_{\nu} + \sum_{d \in \Dcal_Y} a_d d, \quad 1 = \sprod{\varpi, \nu}, \; a_d = \sprod{\varpi, \rho(d)}. \]
%Moreover, \( \varpi = 2 \rho_S - 2 \rho_{S^p} \) where \( 2 \rho \) is the sum of the positive roots of \( G \) and \( 2 \rho_{S^p} \) the sum of the positive roots in \( P^u \), where \( P \) is the left-stabilizer of the open Borel orbit. 

Our first goal is to extend the main result on existence of log Calabi-Yau metrics on toric cones \textit{with an isolated singularity} of de Borbon-Legendre \cite{dBL22} to the spherical context with more general singularities. 

Define 
\[ D := \sum_{\nu \in \Vcal_Y} (1 - \gamma_{\nu}) D_{\nu} \]
to be a \( G \times T_H \)-invariant divisor (which has simple normal crossing support by construction) with \( \gamma = (\gamma_{\nu})_{\nu \in \Vcal_Y} \) satisfying \( 0 < \gamma_{\nu} \leq 1 \) such that the naturally \( G \times T_H \)-linearized divisor
\( -L := K_Y + D \) is \( \Rbb \)-\textit{Cartier}. The latter is equivalent to the existence of \( \varpi_{\gamma} \in \text{int}(\Ccal_Y^{\vee}) \) such that 
\[ \sprod{\varpi_{\gamma}, \nu} = \gamma_v, \; \forall D_{\nu} \in \Vcal_Y, \quad \sprod{\varpi_{\gamma}, \rho(d))} = a_d, \forall d \in \Dcal_Y. \] 
The set of such elements \( \varpi_{\gamma} \) are called \textit{angles}. The pair \( (Y,D) \) is said to be a \textit{spherical log cone} and \( (Y,D, \xi) \) is a \textit{polarized spherical log cone}. Moreover, \( D \) as a closed subvariety is also a \( G\)-spherical cone. 

Given any Reeb vector \( \xi \), one can build a (weak) cone metric \( \omega_{\xi} = \sqrt{-1} \del \delb r_{\xi}^2 \) following \cite{HS16}. We say that a cone metric \( \omega_{\xi} \) on \( Y \) is a \textit{log Calabi-Yau metric with Reeb vector \( \xi \) } if 
\[ \text{Ric}(\omega_{\xi}) = D, \]
which is equivalent to 
\begin{equation}
(\sqrt{-1} \del \delb r_{\xi}^2) = \frac{dV_Y}{ \prod \abs{s_\nu}^{2(1-\gamma_{\nu}) }},
\end{equation}
where \( s_{\nu} \) is the canonical \( G\)-equivariant section of \( D_{\nu} \). 
In particular, \( \omega_{\xi} \) restricts to a bona fide (singular) Ricci-flat Kähler metrics on \( Y \backslash \text{Supp}(D) \). We also expect that \( \omega_{\xi} \) has \textit{conic singularities} of angles \( 2 \pi \gamma _{\nu} \) along \( D_{\nu} \) in the log smooth locus of \( Y \) (conditionally on an analogue of Guenancia-Paun's result \cite{GP16}). 

\begin{thmx}[Prop. \ref{theorem_kstability_spherical_cone}] \label{maintheorem_kstability}
Let \( (Y,D,\xi) \) be a polarized spherical log cone with angles \( \gamma = (\gamma_{\nu})_{\nu \in \Vcal_Y} \), \( 0 < \gamma_{\nu} \leq 1 \), such that \( (Y, D) \) has klt singularities. Then the following are equivalent
\begin{itemize} 
\item \( Y \) has log Calabi-Yau metrics with Reeb vector \( \xi \). 
\item \( (Y,D,\xi) \) is K-stable.
\item \( \text{bar}_{DH}(\Delta_{\xi}) - \varpi_{\gamma} \in (-\Vcal)^{\vee}, \quad \Delta_{\xi} := \set{p \in \Ccal_Y^{\vee}, \sprod{p,\xi} = n} \).
\end{itemize}
%In this case, \( \varpi_{\gamma} \in \Delta_{\xi} \), and the (unique) K-stable Reeb vector can be interpreted as the minimizer of \( \vol_{DH} \). 
\end{thmx}

The K-stability criterion generalizes largely our previous work on horospherical cones \cite{Ngh22b}, which is based on solving an explicit real Monge-Ampère equation through variational approach. Here we explore the algebro-geometric method by constructing explicit \( G\)-equivariant test configurations of a polarized cone via description of \( G\)-equivariant degenerations in \cite{BP87}, \cite{Del20}. Any central fiber of such configuration admits a further equivariant degeneration to a horospherical central fiber \cite{Pop86}, and the Futaki invariant remains constant throughout (Lemma \ref{lemma_futaki_invariant_constant_along_fibers}). 

We then conclude based on an explicit computation of the Futaki invariant of a horospherical cone in Lemma \ref{lemma_futaki_invariant_horospherical}, and the fact that \( G\)-equivariant K-stability over special test configurations is equivalent to K-stability, see Theorem \ref{theorem_ricci_flat_equivalent_k_stability}.

\begin{rmk}
One can compute the generalized \( \delta \)-invariant for spherical log cones, then use the valuative criterion for K-stability in Kai Huang's PhD thesis \cite{Huang22}. Our approach is more geometric in nature and independent of the works in \cite{LLW22} \cite{Y24}.  
\end{rmk}

\begin{rmk}
Based on Pasquier's result on horospherical pairs \cite{Pas16}, we expect that any spherical log pair \( (Y,D) \) defined as above with \( 0 < \gamma_{\nu} < 1 \) has automatically klt singularities. 
\end{rmk}

Our next main result is the following. 

\begin{thmx}[Prop. \ref{prop_stable_degeneration_equals_equivariant_stable_degeneration}] \label{maintheorem_kstabledegeneration}
Any K-semistable spherical log cone \( (W, D, \xi) \) admits an equivariant degeneration to a K-stable spherical log cone \( (C, D_0, \xi) \), unique up to equivariant isomorphisms. 
\end{thmx}

This result might be of independent interest in K-stability theory. In fact, the existence and uniqueness of the \textit{G-equivariant} K-stable degeneration is known in the Fano case \cite{Zh21}, but a proof for log Fano cones is still lacking, since the argument in \cite{Zh21} supposes the existence of a good moduli space for K-(semi)stable Fano varieties, which has not yet been shown to exist for Fano cones in all dimension. 

By the time this article is being prepared, Xu-Zhuang has proved the boundedness property for K-semistable log Fano cones \cite{XZ24}, which is a crucial ingredient for the construction of the moduli spaces.  However, our proof is rather direct and solely based on the combinatorial information in the spherical cone.

\subsection{K-semistable valuations and Calabi-Yau metrics}

Let \( (M,\omega) \) be a \( n\)-dimensional complete Calabi-Yau manifold with \textit{maximal volume growth}, i.e. for every ball \( B_r(p) \) of radius \( r > 0 \) centered as \( p \), there is \( \kappa > 0 \) satisfying 
\[ \text{vol}(B_r(p)) \geq \kappa r^{2n}. \] 
A \textit{metric cone} \( C := C(Z) \) over some compact metric space \( (Z,d_Z) \) is the metric completion of \( ]0,+\infty[ \times Z \) with respect to the metric
\[ d((r_1,z_1),(r_2,z_2)) = \sqrt{r_1^2 + r_2^2 - 2r_1 r_2 \cos \max \set{d_Z(z_1,z_2), \pi}}. \] 
The seminal work of Cheeger-Colding \cite{CC97} shows that, given a sequence \( (M_i, \omega_i) = (M, \lambda_i \omega) \) with \( \lambda_i \to 0 \), after passing to a subsequence we obtain a metric cone \( C \), called the \textit{asymptotic cone} (or \textit{tangent cone at infinity}) of \( M \). 

In the Kähler context, consider the set  \( \Kcal(n,\kappa) \) of complete \( n\)-dimensional polarized Kähler manifolds \( (X,L,\omega,p) \), where \( L \) is a Hermitian holomorphic line bundle over \( X \) with curvature \(- i\omega \) and \( p \) a chosen base point, such that \( (X,\omega) \) is Einstein with Euclidean volume growth as in \cite{DS17}.  

As remarked by Székelyhidi \cite{Sze20}, if we suppose that \( \omega = \sqrt{-1} \del \delb \phi \) for some smooth psh function \( \phi \) on \( M \), then we can readapt the powerful theory of Donaldson-Sun in the noncompact setting by choosing \((L_i, h_i) \) as trivial line bundles over \( M_i \) with Hermitian metric \( h_i = e^{-\lambda_i \phi} \) so that the sequence \( (M_i, L_i, \omega_i) \) lies in the class \( \Kcal(n, \kappa) \). The same arguments in  \cite[Section 3.4]{DS17} then show that the tangent cone at infinity is independent of the chosen subsequence, has a \textit{complex normal affine cone} structure \((C,J_0) \) with \textit{\( \Qbb\)-Gorenstein klt singularities}, and the \textit{metric singular set} of \( C \) (in the sense of Cheeger-Colding) in fact coincides with the algebraic singular set of \( C \) \cite{DS17}.  Moreover, \( (C,J_0) \) has a (weak) \textit{conical Calabi-Yau structure} (so \( C \) is in particular K-stable). 

%This is a \( \del_{J_C}  \delb_{J_C}\)-exact (weak) Ricci-flat metric with potential \( r^2 \), compatible with the complex structure \( J_C \), and homogeneous under the scaling vector field generated by \( r \), i.e. 
%\[ \Lcal_{r \del_{r}} \omega_C = 2 \omega_C. \] 
%In particular, the Reeb vector field \( \xi = - J_C (r \del_{r}) \) associated to \( \omega_C \) generates a holomorphic isometric action of a compact torus \( T_c \) on \( C \) \cite[Lemma 2.17]{DS17}. 

%Note that collapsing behavior may occur without the volume growth assumption, while an equivariant \( \Cbb^{*}\)-degeneration always preserves dimension since the central fiber and generic fibers have the same Hilbert polynomial.The maximal volume growth condition allows us to ignore collapsed Gromov-Hausdorff limits, and the non-collapsed ones correspond to the 2-step degeneration of the Calabi-Yau manifold. 
%The proof will be provided in Prop. \ref{prop_stable_degeneration_equals_equivariant_stable_degeneration}. 

It is generally very hard to classify all the tangent cone at infinity of a given Calabi-Yau affine manifold, even under the maximal volume growth condition. From Donaldson-Sun theory, at least we know that such cone can  be obtained from \( (M,\omega) \) via a \textit{2-step degeneration} as follows.  

First, a complete \( \del \delb\)-exact Calabi-Yau metric \( \omega \) on \( M \) induces a negative valuation \( \nu_{\omega} \) on the ring \( R(M) \) of holomorphic functions with polynomial growth of \( M \) (see Section \ref{section_semistable_valuations_classification} for the precise definition). This valuation moreover induces a filtration on the ring \( R(M) \) and a degeneration of \(  M\) to a K-semistable Fano cone \( (W, \xi) \) with the K-semistable Reeb valuation \( \nu_{\xi} \) induced by \( \nu_{\omega} \) in a natural way. The K-semistable cone \( (W,\xi)\) then degenerates to the K-stable cone \( (C,\xi) \) via a further test configuration. It was recently shown by Sun-Zhang that when \( C \) has smooth link, then \( (M,\omega) \) degenerates to \( (C,\xi) \) in a single step and is moreover \textit{asymptotically conical} in the sense of Conlon-Hein \cite{CH13} \cite{SZ22}. 

\begin{defn} 
We say that the valuation \( \nu_{\omega} \) is \emph{K-stable} (resp. \emph{K-semistable}) if the graded ring of \( R(M) \) by \( \nu_{\omega} \) is finitely generated and defines a \emph{K-stable} (resp. \emph{K-semistable}) Fano cone with the K-stable (resp. K-semistable) Reeb valuation induced by \( \nu_{\omega} \).
\end{defn}

%If \( M \) is quasiprojective, then in general the ring \( R(M) \) does \textit{not} coincide with the coordinate ring of \( M \), but the coordinate ring of an \textit{crepant affine blowdown} of \( M \), cf. \cite{Liu21}. To minimize technical issues, we assume that \( M \) is an affine manifold and that \( \omega \) is \textit{compatible} in the following sense. 

%\begin{defn}
%The metric \( \omega \) on an affine manifold \( M \) is said to be \emph{compatible} if \( R(M) \) is exactly the coordinate ring of \( M \) as an affine variety.
%\end{defn}

In \cite{SZ22}, the authors propose a four-steps scheme to classify complete Calabi-Yau metrics with Euclidean growth in the trivial Kähler class on noncompact manifolds. The scheme consists of 
\begin{enumerate}
    \item Given an affine manifold \( M \), classifying all K-(semi)stable valuations on \( M \). More precisely, given a complete \( \del \delb\)-exact Calabi-Yau metric with maximal volume growth \( \omega \) on \( M \), determine the space of all possible K-(semi)stable valuations on \( M \). 
    \item Given a K-stable valuation \( \nu \) on \( M \), determining the space \( \Mcal_{\nu} \) of all compatible Calabi-Yau metric \( \omega \) on \( M \) such that \( \nu_{\omega} = \nu \). 
    \item  For any \( \omega_1, \omega_2 \in \Mcal_{\nu} \), finding a constant \( c > 0 \) such that 
    \( c^{-1} \omega_2 \leq \omega_1 \leq c\omega_2 \). 
    \item Let \( \Ncal_{\nu} \) be the space of conical Calabi-Yau metrics on the asymptotic cone \( C_{\nu} \). The natural map \( \Mcal_{\nu} \to \Ncal_{\nu} \), defined by taking the rescaled limit of the Kähler form under the weighted cone construction, is bijective. 
\end{enumerate}

Our philosophy, which is rather natural, is that if we impose a large symmetry on the metric, the scheme should be considerably simplified. We thereby achieve Step \( (1) \) for semistable valuations of a Calabi-Yau \textit{spherical manifold}.

\begin{thmx}[Prop. \ref{prop_kstable_valuation_invariant} and \ref{prop_asymptotic_cone_spherical}] \label{maintheorem_kstable_valuation_sphericalmanifold}
If \( M \) is a \( G\)-spherical affine manifold and \( \omega \) is \( K\)-invariant complete Calabi-Yau metric with maximal volume growth in the trivial Kähler class of \( M \), then 
\begin{itemize} 
\item the asymptotic cone \( (C,\xi) \) of \( M \) is a \( G\)-spherical cone and unique up to an isomorphism preserving the K-stable Reeb vector \( \xi \).
\item the negative valuation \( \nu_{\omega} \) is \( G\)-invariant and restricts to the K-stable valuation \( -\nu_{\xi} \) in the Cartan algebra of \( M \) and \( C \). In particular,  there can be only finitely many \( G\)-invariant K-stable valuations on a \( G\)-spherical affine manifold.
\end{itemize}
\end{thmx}
Here are some remarks on this theorem. 
\begin{itemize}
\item An immediate corollary is that the only Calabi-Yau metrics with maximal volume growth and horospherical symmetry are the conical Calabi-Yau metrics on horospherical cones. 
\item The valuation doesn't uniquely determine the Calabi-Yau metric, but only up to a family. An explicit example of a \( 2 \)-parameters family of Calabi-Yau metrics on \( \Cbb^3 \) with asymptotic cone \( \Cbb^2 / \Zbb_3 \times \Cbb \) was constructed by Chiu \cite{C22}. The fourth step in Sun-Zhang classification scheme predicts a family of Calabi-Yau metrics depending on as many parameters as the automorphisms group of the asymptotic cone.
\item As for uniqueness of the asymptotic cone, an approach independent of K-stability theory is to use the equivariant Hilbert scheme constructed by Alexeev-Brion in \cite{AB04a}, generalizing the Haiman-Sturmfels' Hilbert scheme for diagonalizable group \cite{HS04}, used in \cite[Section 3.3]{DS17}. 

\item The \( G\)-invariance of \( \nu_{\omega} \) and uniqueness of \( C \) as an \( G\)-affine cone hold for any \( K\)-invariant Calabi-Yau metric in our context (cf. Remark \ref{remark_general_obstruction}), but it is not clear how to compare \( \nu_{\omega} \) and \( \nu_{\xi} \) as in the spherical case. 
\end{itemize} 

Since every \(  K\)-invariant Calabi-Yau metric on non Hermitian symmetric spaces is necessarily \( \del \delb\)-exact of maximal volume growth, and that the K-stable valuation induced by such metric lies outside the Weyl chamber in the \( G_2 \) case by explicit computations in \cite{Ngh24}, we obtain directly the following non-existence result as announced therein. 

\begin{corx}
There is no complete \(  K\)-invariant Calabi-Yau metric with horospherical asymptotic cone on the symmetric spaces of type \( G_2 \). 

In particular, it follows from \cite{BD19} that there can only be a unique possible \( G\)-spherical asymptotic cone and there exists complete \( K\)-invariant Calabi-Yau metric on \( G_2 \) with this asymptotic cone.
\end{corx}
\begin{itemize} 
\item This provides the first example of a non-rigid singular Calabi-Yau cone that cannot be realized as the tangent cone at infinity of a given \textit{equivariant affine smoothing}. On the other hand, the existence of a AC Calabi-Yau metric on an affine smoothing of a \textit{smooth Calabi-Yau cone} is \textit{always guaranteed} because any affine manifold is Kähler \cite[Theorem A]{CH22}.  
\item Since the \( G_2 \)-cones in the case of multiplicity \( 2 \) have canonical singularities, the non-existence result also suggests that a general existence theorem à la Conlon-Hein \cite[Theorem A,B]{CH22} should involve finer properties of the cone's singularities.

\item Given a Calabi-Yau cone \( C \), it is expected that there are only two ways to obtain Calabi-Yau manifolds: either by smoothing \( C \) or crepantly resolving \( C  \). This turns out to be the case for smooth Calabi-Yau cones \cite{CH22}. If both ways work, one can shrink the exceptional divisor on the Calabi-Yau crepant resolution \( \check{X} \) to \( C  \), then smoothly deforming \( C \) to a Calabi-Yau manifold \( \wh{X} \) with a different complex structure. This phenomenon is called \textit{geometric transition} which is of interest to physicists \cite{CdlO90} \cite{Ros06}. In our context, no \textit{equivariant geometric transition} phenomenon can occur through this cone, since there is no equivariant crepant resolution of the \( G_2 \)-asymptotic cone in the first place (cf. Lemma \ref{lemma_nocrepantresolution}).

\item Ronan Conlon pointed out to me that there is not yet any counterexample when the asymptotic cone has smooth link. It would be interesting to ask whether there exists at all any equivariant Calabi-Yau smoothing of the \( G_2 \)-horospherical asymptotic cones. 
\end{itemize}

Finally, another motivation of our work comes from the author's remark that many known examples of Calabi-Yau manifolds of  maximal volume growth with singular tangent cones so far are in fact \textit{affine spherical manifolds} with respect to the complexified action of the given isometry on the metric. This includes the Li-Conlon-Rochon-Székelyhidi (LCRS) metrics on \( \Cbb^{n+1}, n \geq 2 \) with asymptotic cone \( \Cbb \times A_1 \) \cite{Li19} \cite{CR21} \cite{Sze19}, Biquard-Delcroix-Gauduchon's metrics with horosymmetric tangent cones \cite{BD19}, \cite{BG96}, and the metrics with horospherical tangent cones constructed by the author in \cite{Ngh24}. Note however that on \( \Cbb^{n+1} \), there exist also metrics with non-spherical symmetry \cite{Sze19} \cite{CR21}.

Every \( G \)-affine spherical manifold \( M \) is \(  G\)-isomorphic to \( G \times_H V \), where \( H \) is a reductive connected spherical subgroup of \( G \) (in particular \( G/H \) is an affine spherical space), and \( V \) is a spherical \( H \)-module \cite[Corollary 2.2]{KvS06}. 

\begin{ex}
The complex symmetric spaces \( G/H \) are all affine spherical manifolds. On the other hand, \( \Cbb^{n+1} \) is a rank two \( \SO_{n+1} \times \Cbb^{*} \)-nonsemisimple symmetric cone with open orbit \( \SO_{n+1}/\SO_{n} \times \Cbb^{*} \). 
\end{ex}

The LCRS metrics with spherical symmetry on \( \Cbb^{3} \) are invariant by the maximal compact subgroup \( K = \SO_{n+1}(\Rbb) \times \Sbb^1 \) and of horospherical tangent cones at infinity \( A_1 \times \Cbb \). By Székelyhidi's uniqueness theorem \cite{Sze20}, any complete Calabi-Yau metric on \( \Cbb^{3}\) asymptotic to the cone is unique up to scalings and biholomorphisms. Note that any \(  K\)-invariant metric on the symmetric cone \( \Cbb^3 \) is \( \del \delb\)-exact and has maximal volume growth (cf. \cite{Del20b}). This fact combined with Székelyhidi's uniqueness and Theorem \ref{maintheorem_kstable_valuation_sphericalmanifold} implies the following. 

\begin{corx}[Prop. \ref{proposition_calabiyaumetrics_c3}] \label{main_corollary_calabi_yau_c3}
%The tangent cones at infinity of a quasiprojective \( G\)-spherical variety \((M, \omega) \) are either 
%\begin{itemize}
%    \item the horospherical cone corresponding to the degeneration of \( M \) along the unique K-stable valuation in the interior of \( \Vcal \).   
%    \item or the cones corresponding to the degeneration of \( M \) along the K-semistable valuations lying on \(  \del \Vcal \) (after possibly a further degenerations).  
%\end{itemize}
%It turns out that the horospherical K-stable valuation may not always exist, for example on \( G_2 \)-symmetric spaces. However on symmetric spaces up to rank two, a K-stable valuation always exists. We expect that there is always at least one K-stable valuation on a quasiprojective spherical manifold, and if we admit the no semistability picture of Sun-Zhang, then the K-semistable valuations on \( \del \Vcal \) (if any) are actually K-stable. 
The only possible asymptotic cones of complete Calabi-Yau metrics with spherical symmetry on \( \Cbb^{3} \) are 
\begin{itemize} 
\item the horospherical asymptotic cone \( A_1 \times \Cbb \) of the LCRS metrics,
\item and the asymptotic cone \( \Cbb^{3} \) itself of the standard flat metric. 
\end{itemize}
In particular, there are only two distinct families of complete Calabi-Yau metrics with spherical symmetry of \( \Cbb^{3} \). 
\end{corx}

%Every \(  K\)-invariant Calabi-Yau metric on a spherical manifold, if exists, can thus be interpreted as an equivariant smoothing of a singular Calabi-Yau spherical cone (possibly followed by a crepant resolution). 

\subsection{Organization.}
The paper is organized as follows. In Section \ref{section_test_configuration_futaki_invariant}, we describe the test configurations and compute the Futaki invariant of spherical cones. Main Theorems \ref{maintheorem_kstability} and \ref{maintheorem_kstabledegeneration} are proved in Section \ref{section_proof_kstability_maintheorems}. 

Section \ref{section_semistable_valuations_classification} contains a summary of Donaldson-Sun theory. The proof of Theorem \ref{maintheorem_kstable_valuation_sphericalmanifold} is given in Section \ref{section_proof_kstable_valuations_maintheorem}. Examples of explicit K-stable valuations on spherical Calabi-Yau manifolds and proof of Corollary \ref{main_corollary_calabi_yau_c3} are given in Section \ref{section_examples_kstable_valuations}. 

\textbf{Acknowledgement.}
This paper is part of a thesis prepared under the supervision of Thibaut Delcroix and Marc Herzlich. I am thankful to Thibaut Delcroix for a lot of helpful exchanges and his careful reading of the paper. Thanks also go to Sébastien Boucksom for patiently listening to me and kindly sharing his insights, as well as Ronan Conlon and Eveline Legendre for comments related to various parts of the paper. The author is partially supported by ANR-21-CE40-0011 JCJC project MARGE.

\section{Test configurations and Futaki invariant of spherical cones} \label{section_test_configuration_futaki_invariant}

\subsection{Generalities on spherical cones}
Main references for this section are \cite{Bri97}, \cite{Kno91}. 
A spherical space is a homogeneous space \( G/H \) containing a Zariski-open orbit under the action of a Borel subgroup \( B \subset G \). A \( G \)-spherical variety \( X \) is a \( G\)-equivariant embedding of a spherical space. A spherical variety is said to be \textit{simple} if it contains a unique closed \( G\)-orbit. Each simple spherical variety contains an open \( B\)-stable affine subset \( X_B \) that intersects the closed orbit along an open \( B \)-stable orbit. Every spherical variety can be covered by simple spherical varieties. 

Let \( \Mcal(G/H) \) be the lattice of characters of \( \Cbb(G/H) \) as a \( B\)-representation, and \( \Ncal(G/H) \) be its dual lattice. Denote by \( \Vcal(G/H) \) the set of \( G\)-invariant valuations on \( \Cbb(G/H)^{*} \). When the spherical space is clear from the context, we will just denote them by \( \Mcal, \Ncal, \Vcal \).  

\begin{thm} \cite{BLV}
Let \( Q \) be the parabolic subgroup of \( G \) that stabilizes the open \( B\)-orbit \( BH \) (or equivalently, stabilizes all the colors in \( \Dcal \)). 

There is a choice of a Lévi subgroup \( L \subset Q \) and of a maximal torus \( T \subset L \) (this is also the maximal torus of \( G \)) such that one can identify \( \Mcal \) and \( \Ncal \) with the character lattice of the \emph{adapted torus} \( T / T \cap H \) and its dual. The dimension of this torus is the \emph{rank} of \( G/H \). 
\end{thm}

For every valuation \( \nu \in \Vcal \), there exists an injective natural map \( \rho : \Vcal \to \Ncal_{\Rbb} \), such that \( \rho(\nu)(f_{\chi}) = \sprod{\chi, \nu} \) where \( f_{\chi} \in \Cbb(G/H) \) is an eigenvector of \( B\) with character \( \chi \). 

\begin{defn}
The set of reduced and irreducible \( B\)-stable divisors in \( G/H \) is called the \emph{colors of \( G/H \)}, denoted by \( \Dcal \). 

A \emph{color of a \( G/H \)-spherical embedding \( X \)} is an element of \( \Dcal \) whose closure in \( X \) contains a closed orbit. The set of colors of a spherical embedding \( X \) is denoted by \( \Dcal_X \). The natural map \( \rho \) sends \( \Dcal \) to a subset of \( \Ncal \), but \( \rho \) is not injective on \( \Dcal \) in general.

Let \( \Vcal_X \) be the set of \( G\)-invariant divisors of \( X \). The injective map \( \rho : \Vcal_X \to \Ncal_{\Rbb} \) that sends a divisor to its valuation identifies \( \Vcal_X \) with a finite subset in \( \Vcal \). 
\end{defn}

To each simple embedding \( X \), we can associate a pair \( (\Ccal_X, \Dcal_X) \), where \( \Ccal_X \) is the strictly convex cone generated by \( \Vcal_X \cup \rho(\Dcal_X) \), called a \textit{colored cone} in the following sense. 

\begin{defn}
A \emph{colored cone} \( (\Ccal, \Fcal) \) is the data of \( \Ccal \subset \Ncal_{\Rbb} \) and \( \Fcal \subset \Dcal \), where \( 0 \notin \rho(\Fcal) \),  \( \Ccal \) is a strictly convex cone generated by \( \rho(\Fcal) \) and a finite number of elements of \( \Vcal \), and \( \Fcal \) is called the set of \emph{colors} of \( (\Ccal, \Fcal) \). 
\end{defn}

\begin{thm}[Luna-Vust]
The map \( X \to (\Ccal_X, \Dcal_X) \) is a bijection between the set of isomorphism classes of simple \( G/H \)-embeddings and the set of colored cones. 
\end{thm}

\begin{thm} \cite{Ngh22b}
Let \( G/H \) be a spherical space. Let \( Y \) be a simple \( G\)-equivariant embedding of \( G/H \) with colored cone \( (\Ccal_Y, \Dcal_Y) \). Then \( Y \) is a spherical cone if and only if \( \Vcal(G/H) \) has a linear part, \( \Ccal_Y \) is of maximal dimension, and \( \Dcal = \Dcal_Y \) (i.e. all the colors of \( G/H \) contains the unique closed orbit of \( Y \)).  
\end{thm}

The \( \Qbb\)-Gorenstein assumption on a spherical cone implies that is has at worst klt singularities. We refer the reader to \cite{Pas17} for a survey on singularities of spherical varieties.  Any \( \Qbb\)-Gorenstein spherical cone is in particular a Fano cone. 

%Each spherical cone \( Y \)  can be identified with its \( B\)-stable affine chart \( Y_B \). 

\begin{thm} \cite{Los09}
Let \( \Gamma_Y := \Ccal_Y^{\vee} \cap \Mcal \) be the weight monoid of the spherical cone \( Y \). Then \( Y \) is uniquely determined up to \( G\)-isomorphisms by \( (\Gamma_Y, \Sigma_Y) \). 
\end{thm}

%\textcolor{red}{One can embed a spherical cone of rank \( r \) explicitly in \( \Cbb^N \) as follows. The monoid \(  \Ccal_Y^{\vee} \cap \Mcal \) is saturated (since \( Y \) is normal) and finitely generated, and its generators can be identified with characters \( \chi_1, \dots, \chi_s \) of the adapted torus \( T / T \cap H  \). The open affine chart \( Y_B \) is then isomorphic to the affine variety with coordinate ring \( \Cbb[x_1, \dots, x_r]/ \sprod{\chi_1, \dots, \chi_s} \), which is also the coordinate ring of \( Y \).}  

Let \( T_H := \text{Aut}_G(Y)^0 \simeq (N_G(H)/H)^0 \) be the neutral component of the automorphisms of \( Y \) that commutes with \( G \). Since every \( \sigma \in \text{Aut}_G(G/H)^0 \) can be extended to a \( G\)-equivariant isomorphism of \( (Y, y)  \) to \( (Y,\sigma(y)) \), we have \( T_H \simeq \text{Aut}_G(G/H)^0 \). Moreover, \( \dim (T_H) = \dim \text{lin} \Vcal \geq 1 \), and the noncompact Lie algebra of \( \tfrak_H \) can be identified with \( \text{lin} \Vcal \), hence \( \Ncal(T_H) = \text{lin} \Vcal \cap \Ncal \). 

\begin{ex}
Every toric space (i.e. \( G = T \) and \( H = \set{1} \)) admits conical embeddings, while this is not the case for every spherical space. Indeed the symmetric space \( SL_2 /T \) does not embed into any symmetric cone, since \( N_{SL_2}(T)/T \simeq \Zbb_2 \). However, the space \( SL_2 / T \times \Cbb^{*} \) has a conical embedding. 
\end{ex}

Under the \( T_H \)-action, the coordinate ring of \( Y \) decomposes as
\[ R := \Cbb[Y] = \bigoplus_{\alpha \in \Gamma} R_{\alpha}, \]
where \( \Gamma := \set{ \alpha \in \Mcal(T_H), R_{\alpha} \neq 0} \) is a finitely generated monoid in \( \Mcal(T_H) \). The cone \( \sigma^{\vee} \) generated by \( \Gamma \) is strictly convex and of maximal dimension in \( \Mcal(T_H)_{\Rbb} \).  By duality, the dual cone \( \sigma = (\sigma^{\vee})^{\vee} \) is also a strictly convex cone of maximal dimension. 

\begin{rmk} \label{remark_bmodule_tmodule}
Since the right action of \( T_H \) commutes with \( G\), every \( B\)-module \( R^{(\alpha)} \) can be identified \emph{as a one-dimensional \( \Cbb\)-vector space} with a \( T_H \)-module \( R_{\alpha_H} \) such that \( \alpha|_{\Ncal(T_H)} = \alpha_H \). 
\end{rmk}

\begin{defn} \label{definition_Reeb_cone}
The interior of \( \sigma \) is called the \emph{(algebraic) Reeb cone} of \( Y\), denoted by \( \Ccal_R \). A couple \( (Y,\xi)\) with \( \xi \in \Ccal_R \) is said to be a \emph{polarized cone}.

An element \( \xi \in \Ncal(T_H)_{\Qbb} \) is said to be \emph{quasi-regular}, and \emph{irregular} otherwise. 

Every Reeb vector induces a monomial valuation \( \nu_{\xi} \) on \( \Cbb[Y] \), centered on the unique fixed point of \( Y \), such that 
\[ \nu_{\xi}(f) = \min_{\alpha \in \Gamma} \set{ \sprod{\alpha, \xi}, R_{\alpha} \neq 0}. \]
\end{defn}

Note that when the cone \( Y \) has smooth link, then the algebraic Reeb cone can be identified with the symplectic Reeb cone as follows. Let \( J \) be a complex structure on \( Y^{*} := Y \backslash \set{0} \). A Kähler metric \( \omega \) on \( (Y,J) \) is compatible with a Reeb element \( \xi \in \Ccal_R \) if there exists a \( \xi \)-invariant smooth psh function \( r : Y^{*} \to \Rbb_{>0} \) such that \( \omega = \frac{1}{2} i \del \delb r^2 \) and \( \xi = J(r \del r) \).  

Given a quasi-regular Reeb vector field \( \xi_0 \in \Ncal(T_H)_{\Qbb} \), it can be shown that \( Y \) always admits a \( \xi_0 \)-compatible metric and a dual 1-form \( \eta_0 \) on \( Y^{*} \) such that \( \eta_0(\xi_0) = 1 \). In this case, the \textit{symplectic Reeb cone} \[ \Ccal_{R}' := \set{\xi \in \tfrak_H, \eta_0(\xi) > 0  \; \text{on} \; Y^{*}} \]
turns out to be exactly the algebraic Reeb cone \( \Ccal_R \) (cf. \cite[Prop. 2.7]{CS18}). In particular, it is independent of the choice of \( \xi_0 \) and \( \eta_0 \).

\begin{defn} \label{definition_horospherical}
A spherical space \( G/H \) is called \emph{horospherical} if \( H \) contains a maximal unipotent subgroup of \( G \).
\end{defn} 

Let \( S \) the set of simple roots of \( G \) with respect to a Borel subgroup \( B\) and \( W \) the Weyl group of \( G \). Recall that there is a bijection between the subsets of \( S \) and (the conjugacy classes of) parabolic subgroups of \( G \) as follows. For every \( I \subset S \), let \( W_I \) the subgroup of \( W \) generated by the reflection \( s_{\alpha}, \alpha \in I \).  The parabolic subgroup \( P_I \subset G \) is defined as the group generated by \( B\) and \( W_I \).

Given a dominant weight \( \lambda\), we have \( \lambda = \sum_{\alpha \in S} x_{\alpha} \varpi_{\alpha}, x_{\alpha} \geq 0 \). Then define the parabolic subgroup \( P(\lambda) \) as \(P_I \), where \( I = \set{\alpha \in S, x_{\alpha} = 0}. \) In particular, \( P(\varpi_{\alpha}) = P_{S \backslash \set{\alpha}} \), and 
\[ \cap_{\alpha \in S \backslash I} P(\varpi_{\alpha}) = P_I. \]

\begin{prop} \cite{Pas06} \label{proposition_horospherical_space}
A \( G \) horospherical space is uniquely determined by a couple \( ( \Mcal, I) \) where \( I \subset S \), and \( \Mcal \) is a sublattice of \( \Mcal(T) \) such that for all \( \chi \in \Mcal \) and \( \alpha \in I \), \( \sprod{\chi, \alpha^{\vee}} = 0 \). 
The isotropy subgroup is then 
\[ H = \cap_{\chi \in \Mcal(P_I)} \text{ker}(\chi). \] 

Furthermore, \( P_I \) is the right-stabilizer of the open Borel orbit, and coincides with \( N_G(H) \), and \( G/H \) is an equivariant torus bundle over \( G/ P_I \) with fiber the torus \( P_I /H \). The colors \( \Dcal \) of \( G/H \) are in bijection with the roots in \( S \backslash I \) and
\[ \rho(\Dcal) = \set{\alpha^{\vee}|_{\Mcal_I}, \alpha \in S \backslash I}. \] 
\end{prop}

Note that \( P_I \) is the opposite parabolic subgroup of the (left-)stabilizer \( Q \). When \( Y \) is a conical embedding of a horospherical space \( G/H \) with colored cone \( (\Ccal_Y, \Dcal_Y) \), the group \( T_H \) coincides with \( P_I /H \), but since the action is reverse, the Reeb cone is exactly 
\[ \Ccal_R = - \text{int}(\Ccal_Y). \]

Horospherical cones can be obtained systematically as follows. 

\begin{prop} \label{proposition_horospherical_construction} \cite[Theorem 1]{PV72}
Let \( V(\lambda) \) be a simple \( G \)-module of highest weight \( \lambda \) and eigenvector \( v_{\lambda} \). The variety 
\[ X(\lambda) := \ol{G v_{\lambda}} \subset V(\lambda) \]
is then a rank one horospherical cone over the corresponding Grassmannian \( G/ P(\lambda) \)  in \( \Pbb(V(\lambda)) \) where \( I = \set{\alpha \in S, \sprod{\lambda, \alpha^{\vee}} = 0} \) and \( P_I = P(\lambda) \) is the stabilizer of \( [v_{\lambda}] \in \Pbb(V(\lambda)) \). Moreover, \( \Cbb[X] \simeq V(\lambda)^{*} \). 
\end{prop}

\begin{ex} \label{example_horospherical}
As an application, one can take \( G = \SO_3 \) with the unique fundamental weight \( \lambda = 2 \omega \), where \( \omega\) is the fundamental weight of \( \SL_2\). Then \( X(2 \omega) \) is isomorphic to the ordinary double point, which is the Stenzel asymptotic cone of the rank one symmetric space \( \SO_3/ \SO_2 \). Indeed, \( V(2\omega)^{*} \simeq S^2 V^{*} \simeq \Cbb[x^2, xy, y^2 ] \), which is the coordinate ring of the ordinary double point. 
On the other hand, \( X(\omega) \) is simply \( \Cbb^2 \). 
\end{ex}

\subsection{Test configurations of spherical cones}
%For simplicity, we only describe the test configuration of a klt pair \( (Y,D) \) with \( D = \varnothing \), i.e. a \( \Qbb\)-Gorenstein spherical cone \( Y \). 

Recall that by a result of Knop \cite{Kno91}, if the vertex of a Fano cone \( Y \) is fixed by a reductive group \( G \) acting effectively on \( Y \), then there is a \( \Cbb^{*} \)-action on \( Y \) commuting with \( G\).

\begin{defn} 
Let \( (Y, D, \xi) \) be any polarized log Fano cone, endowed with an effective action of a reductive group \( G \) that fixes the vertex of \( Y\), and a compatible action of a complex torus \( T \) containing \( T_{\xi} \), preserving \( D\). 

A \( G \times T \)-equivariant \textit{test configuration} of \( (Y,D,\xi) \) consists of
\begin{itemize} 
\item a \( G \times T \)-equivariant flat affine family \( \pi : (\Ycal, \Dcal) \to \Cbb \), where \( \Dcal \) is an effective divisor not containing any component of \( Y_0 = \pi^{-1}(0) \) such that each fiber away from \( 0 \) is isomorphic to \( (Y,D) \). 
\item a \( \Cbb^{*} \)-holomorphic action on \( (\Ycal, \Dcal) \) generated by \( \zeta \in \tfrak \) and commuting with the \( G \times T \)-action such that \( \pi \) is \( \Cbb^{*} \)-equivariant for this action, and that there is a \(G \times \Cbb^{*}\)-equivariant isomorphism \( (\Ycal, \Dcal) \backslash (Y_0,D_0) \simeq (Y,D) \times \Cbb^{*} \). 
\end{itemize}
The test configuration is said to be \emph{special} if \( K_{\Ycal} + \Dcal \) is \( \Rbb\)-Cartier and that the central fiber \( (Y_0, D_0) \) is a klt pair. 

Finally, the test configuration is said to be \emph{trivial} if there is a \( T \)-equivariant isomorphism \( (\Ycal,\Dcal) \simeq (Y,D) \times \Cbb \) and \( \zeta = \zeta_0 + t \del_t \) where \( \zeta_0 \) generates a \( \Cbb^{*}\)-holomorphic vector field that commutes with the action of \( \xi \), and \( t \) is an element of the compact Lie algebra of \( \Cbb \). 
\end{defn}

\begin{rmk}
It is well-known by Hironaka's lemma \cite[II.9.12]{Har77} that the configuration \( (\Ycal,\Dcal)  \) is itself a klt pair if the central fiber \( (Y_0, D_0) \) is. 
\end{rmk}

%Using flatness of \( \pi \), one can show that 
%\begin{lem} \label{lemma_coordinate_ring_invariant_along_fiber}
%For every \( t, t' \in \Cbb\), the rings \( \Cbb[Y_t] \) and \( \Cbb[Y_{t'}] \) are isomorphic as \( T \)-representations. In particular, \( Y\) and \( Y_0 \) have the same weight and Reeb cones. 
%\end{lem}

%\begin{proof}
%This seems to be well-known for experts. For lack of reference, we include a quick proof. A special test configuration of \( (Y, \xi) \) induces a flat projective family \( (\ol{\Ycal}, \Lcal) \) (which can be seen as the closure of \( \Ycal \) inside a projective space). The action of \( T \) then extends to \( \Ycal \) be Sumihiro's equivariant completion theorem. After replacing \( \Lcal \) with a multiple, one can assume that \( \Lcal \) is \( T \)-linearized. It is then well-known that \( H^0( \ol{Y}_{t}, \Lcal_t ) \) is isomorphic to \( H^0(\ol{Y}_{t'}, \Lcal_{t'}) \) as \( T  \)-representations (see e.g. \cite{Del20}). But since \( H^0(\ol{Y}_t, \Lcal_t)|_{Y_t} = \Cbb[Y_t] \) is \( T \)-stable, it follows that \( \Cbb[Y_t] \) is also isomorphic to \( \Cbb[Y_{t'}] \) as \( T \)-representations. The lemma follows. 
%\end{proof}

%In concrete terms, a \( G\)-equivariant test configuration can be constructed by embedding \( Y \) into \( \Cbb^N \) in a \( G\)-equivariant manner using the generators of the monoid \( \Ccal_Y^{\vee} \cap \Mcal \), then acting on \( \Cbb^N \) by a \( G\)-compatible \( \Cbb^{*} \)-action. 

In the spherical context, if the test configuration is special, then the central fiber \( (Y_0, D_0, \xi) \) is also a polarized \( G \)-spherical log cone that inherits an action of \( T_H \) and a new action of \( \Cbb^{*} \) that commutes with \( G \times T_H \). The action of the automorphism group \( \text{Aut}_{G \times \Cbb^{*}} (G/H \times \Cbb^{*})^0 \supset T_H \times \set{1} \) extends automatically on \( (\Ycal, \Dcal) \), hence a \( G \)-equivariant test configuration of \( (Y,D,\xi) \) is also a \( G \times T_H \)-test configuration for \( (Y,D,\xi) \). Moreover, since it suffices to check K-stability over special test configurations (cf. Theorem \ref{theorem_ricci_flat_equivalent_k_stability}), we will mainly focus on \textit{special} \( G\)-equivariant configurations. 

\begin{defn}
An \emph{elementary embedding} is a \( G\)-equivariant embedding of \( G/H \) with a unique closed orbit of codimension \( 1 \). A \emph{\( \Cbb^{*} \)-equivariant degeneration} of \( G/H \) is a \( G \times \Cbb^{*} \)-equivariant elementary embedding \( E \) of \( G/H \times \Cbb^{*} \) together with a \( \Cbb^{*}\)-equivariant morphism \( E \to \Cbb \). 
\end{defn}

Every couple \( (\lambda, m) \in \Vcal \oplus \Qbb^{*} \) determines an equivariant degeneration, and vice versa: a primitive generator of the colored cone of \( E \) is of the form \( (\lambda, m) \in \Vcal \oplus \Qbb^{*} \). 
The closed orbit of \( E \) can be identified with \( G/H_0 \), where \( H_0 \) is a spherical subgroup of \( G \). If \( \lambda \in \text{int}(\Vcal) \), then \( G/H_0 \) is horospherical. Moreover, \( G/H_0 \) has the same left-stabilizer of the open Borel-orbit as well as the same adapted Levi subgroup as \( G/H \). 

For simplicity, we only describe here the test configuration of a polarized spherical cone, as the description for a log pair follows almost word-by-word. We first need the following result on spherical morphisms. 

\begin{prop} \cite{Kno91} \label{prop_spherical_morphisms}
There exists a morphism between two \( G/H \)-embeddings \( X \) and \( X' \) if and only if for every colored cone \( (\Ccal, \Fcal) \) of \( X \), there is a colored cone \( (\Ccal', \Fcal') \) of \( X'\) such that \( \Ccal \subset \Ccal' \) and \( \Fcal \subset \Fcal' \). 
\end{prop}

\begin{thm} \label{theorem_test_configurations_spherical_cone}

Let \( (Y,\xi) \) be a polarized \( \Qbb\)-Gorenstein spherical cone. 
\begin{enumerate} 
\item To each \( G\)-equivariant special test configuration of \( (Y,\xi) \) with \( G \)-spherical central fiber \( Y_0 \), there exists \( (\nu,m) \in  \Vcal \oplus \Nbb^{*} \) and a spherical subgroup \(H_0 \subset G \) such that \( Y_0 \) is a \( G/H_0 \)-spherical embedding, and that the action of \( \Cbb^{*} \) on \( G/H_0 \) is 
\[ e^{\tau}. gH_0 = g \nu(e^{-\tau/m}) H_0. \]
\item Conversely, let \( \nu \in  \Vcal \) and \( m \in \Nbb^{*} \). Let \( G/H_0 \) be the central fiber of the equivariant degeneration induced by \( (\nu,m) \). Then there exists  a \( G\)-equivariant test configuration (and a special one after a suitable base change) such that the central fiber \( Y_0 \) is a conical embedding of \( G/H_0 \), and that the \( \Cbb^{*} \)-action can be described as above. 

In particular, every polarized \( G \)-spherical cone admits a test configuration with \( G\)-horospherical central fiber. 
\item Up to lattice isomorphisms, the lattices and weight monoids of \( Y \) and \( Y_0 \) are the same. 
\item \( (\Ycal,\xi; \nu) \) is trivial if and only if \( \nu \) belongs to the linear part of \( \Vcal \).
\end{enumerate}
\end{thm}

\begin{proof}
Before going through the proof, remark that the spherical space \( G/H \times \Cbb^{*} \) has character lattices \( \Mcal(G/H \times \Cbb^{*}) = \Mcal \oplus \Zbb \) and valuation cone \( \Vcal(G/H \times \Cbb^{*}) = \Vcal \oplus \Rbb \), which clearly has non-trivial linear part. The colors of \( G/H \times \Cbb^{*} \) are exactly \( \set{d \times \Cbb^{*}, d \in \Dcal} \). 

Every test configuration induces a \( \Cbb^{*} \)-equivariant degeneration of \( G/H \), hence there exists \( \lambda \in \Vcal \) and \( m \in \Nbb^{*} \) such that the ray generated by \( (\lambda, m) \in \Vcal \oplus \Nbb^{*} \) is the colored cone of the equivariant degeneration. The \( \Cbb^{*} \)-action on \( G/H_0 \) is described in \cite{BP87}, \cite{Del20}. Moreover this action commutes with \( \xi \), so this action must lie in \( \text{lin} \Vcal \cap \Ccal_R \). 

Conversely, let \( (\nu, m) \in \Vcal \oplus \Nbb^{*} \) and consider the conical embedding defined by the generators of \( \Ccal_Y \) and \( ( \nu, m) \), with all the colors of \( G/H \times \Cbb^{*} \). This defines clearly a \( G \times \Cbb^{*} \)-spherical cone \( \Ycal \), and the projection to \( \Rbb_{\geq 0}(0,1) \) with \( 0 \in \Ncal \) in the Euclidean spaces gives an affine \( G \times \Cbb^{*}\)-equivariant morphism \( \pi : \Ycal \to \Cbb \) by classification of spherical morphisms recalled in Prop. \ref{prop_spherical_morphisms}.   

The central fiber \( Y_0 \) corresponds then to the divisor of \( \Ycal \) determined by the ray \( (\nu, m) \). The latter can also be seen as an elementary embedding of \( G/H \times \Cbb^{*} \), hence an equivariant degeneration of \( G/H \) to \( G/H_0 \). The special test configuration is obtained after changing the lattice \( \Ncal \oplus \Zbb \) to \( \Ncal \oplus \frac{1}{k}  \Zbb \) for a suitable \( k \), while keeping the colored cone of \( \Ycal \). 

Remark that the coordinate rings \( R, R_0 \) of \( Y, Y_0 \) are isomorphic as \( G\)-modules, hence \( \Mcal \simeq \Mcal_0 \). Furthermore, \( R^{(B)} \) is \( B\)-isomorphic to \( R_0^{(B)}\) \cite[Proposition 4]{Pop86}, hence \( \Gamma_Y \simeq \Gamma_{Y_0} \). 

Finally, taking \( \nu \) that projects to the interior of \( \Vcal \) then yields a test configuration with horospherical central fiber. The last statement results from \cite{Del20}. 
\end{proof}

We will denote from now on \( ( \Ycal, \xi; \nu) \) the \( G\)-equivariant test configuration of \( (Y, \xi) \) with respect to \( \nu \in \Vcal \).

\begin{rmk} \label{remark_central_fiber_data}
The embedding data of the central fiber \( Y_0 \) of \( (\Ycal, \xi; \nu) \) can be obtained as follows. The weight lattice \( \Mcal_0 \) of \( Y_0 \) can be identified with
\[ \Mcal_0 := (\nu^{\perp} \cap \Mcal) \oplus \Zbb \chi \simeq \Mcal,  \]
where \( \chi \in \Mcal \) is such that \( \sprod{\chi, \nu} = 1 \). In particular, if we let \( \pi : \Ncal \to \Ncal_0 \) be the dual map of the isomorphism \( \Mcal_0 \simeq \Mcal \), then 
\[ \Vcal_0 = \Rbb \nu \oplus \pi(\Vcal). \] 
Since the weight monoids of \( Y \) and \( Y_0 \) are the same, their colored cones have the same support, and the colors of \( Y_0 \) can be determined using \cite{GH15}. 
\end{rmk}

\subsection{Futaki invariant}

Let us recall briefly the construction of Futaki invariant by Collins-Székelyhidi \cite{CS18} through index character and the equivalent characterization of Li-Wang-Xu in terms of normalized volume and log discrepancy \cite{LWX}. 
Let \( (Y,\xi) \) be a \( n\)-dimensional polarized spherical cone and
\[ \Cbb[Y] = \oplus_{\alpha \in \Gamma}  R_{\alpha} \] 
be the decomposition of \( \Cbb[Y] \) as a \( T_H \)-representation. For any \(t \in \Cbb\) and \( \xi \in \Ncal_{\Rbb} \), the index character is defined as
\[F(t,\xi) := \sum_{\alpha \in \Gamma} e^{-t \sprod{\alpha,\xi}} \dim R_{\alpha}. \]
This is a meromorphic function on \( \Cbb\) with poles along imaginary axis, and decomposes near \( t = 0 \) as 
\[ F(t,\xi) = \frac{a_0(\xi) n! }{t^{n+1}} + \frac{a_1(\xi)(n-1)!}{t^{n}} + O(t^{1-n}). \]
where \( a_0, a_1 : \Ccal_R \to \Rbb \) are smooth functions. 

Let \( d_{\xi} f (\nu) \) be the directional derivative of a function at a point \( \xi \) along the vector \( \nu \). The Futaki invariant of the test configuration \( (\Ycal, \xi; \nu) \) is defined by 
\[ \Fut_{\xi}(\Ycal, \nu) = \frac{a_0(\xi)}{n} d_{\xi} \tuple{\frac{a_1}{a_0}}(\nu) + \frac{a_1(\xi) d_{\xi}a_0(\nu) }{n(n+1) a_0(\xi)}. \]
In particular, the Futaki invariant of a test configuration depends only on the coordinate ring of the central fiber as a representation of \( T_H \). For computational reason, we shall use the definition of the Futaki invariant by Li-Wang-Xu in terms of normalized volume and log discrepancy, but note that this is the Futaki invariant of \cite{CS18} up to a positive constant (see for example \cite{LLX20} for details).

Let \( Y \) be a klt normal variety. The \textit{log discrepancy function} of \( Y \) is a positive function \( A_Y \) over the set of valuations that admit a center on \( Y \). For practical reason, we only give the definition of the log discrepancy for a divisorial valuation and refer the reader to \cite[Theorem 2.2, Theorem 3.5]{LLX20} for the general definition for a pair \( (Y,D) \).  Let \( E \) be the exceptional divisor over a proper birational model \( \mu : Y' \to Y \), and \( w_E \) the associated valuation over \( \Cbb(Y') = \Cbb(Y) \), the log discrepancy is then  
\[ A_Y(w_E) := 1 + w_E (K_{Y'} - \mu^{*} K_Y). \]
The general discrepancy for a quasimonomial valuation is then defined in an obvious way, and for a general valuation centered on \( Y \) by using the retraction map from \( \text{Val}_Y \) to the set of quasimonomial valuations over any log smooth model of \( Y \).  
%Given a \( \Qbb\)-Gorenstein Fano cone \(Y \), the log discrepancy \( A_Y \) can be computed in terms of the canonical section. It can be moreover characterized by a linear function on the Reeb cone \( \Ccal_R \). In the horospherical case, this linear function is exactly the opposite of the linear function on \( (\Ccal_Y, \Dcal_Y) \) given by the \( \Qbb\)-Gorenstein condition, as we will show later during our computation of the Futaki invariant. 

\begin{prop} \cite{CS19}, \cite{Li18} \label{prop_ricci_flat_reeb_normalized}
Let \( (Y,D) \) be a spherical log cone with angles \( \gamma \) and \( m \) be an integer such that \( m(K_Y+D) \) is Cartier. Let \( s \) be a \( G \times T_H\)-equivariant nowhere-vanishing holomorphic section of \( -m(K_Y+D) \). Then there exists a linear function \( \varpi_{\gamma} : \Ccal_R \to \Rbb \) such that \[ \Lcal_{\xi} s = m \sprod{\varpi_{\gamma},\xi} s. \]
Moreover, the log discrepancy of \( w_{\xi} \) is exactly
\[ A_{(Y,D)}(w_{\xi}) = \sprod{\varpi_{\gamma}, \xi}. \] 
If \( (Y,D,\xi) \) has log Calabi-Yau cone metrics, then \( A_{(Y,D)}(w_{\xi}) = n \). 
\end{prop}

\begin{defn}
Let \( (Y,D) \) be a log spherical cone. Let \( \xi \) be an element in the Reeb cone \( \Ccal_R \) and \( \varpi_{\gamma} : \Ccal_R \to \Rbb \) the linear function as above. The \emph{(algebraic) volume} of \( (Y,\xi) \) is defined as
\[ \vol_{Y}(\xi) = \lim_{k \to \infty} \frac{ \dim \tuple{ \bigoplus_{ \sprod{\alpha, \xi} < k} R_{\alpha}}}{k^n / n!}. \] 
The normalized volume of a spherical log cone \( (Y,D) \) is a function that takes \( \xi \in \Ccal_R \) to 
\[ \wh{\vol}_{(Y,D)}(\xi) := A_{(Y,D)}(w_{\xi})^n \vol_Y(\xi) = \sprod{\varpi_{\gamma}, \xi}^n \vol_Y(\xi). \]
\end{defn}

\begin{rmk} 
It has been established that \( \vol \) is a continuous function, see e.g. \cite[Theorem 4.10]{CS18} for the case where \( Y \) is smooth, and \cite[Theorem 2.8]{Li21} for the general case. From the differentio-geometric point of view, \( \vol \) is the \( g\)-weighted volume of the (log) Fano base. More precisely, when \( \xi \) is quasi-regular, \( \vol_Y(\xi) \) is exactly the volume of the quasi-regular quotient with respect to the transverse Kähler form for \( \xi \). 
\end{rmk}

\begin{defn}
Let \( (\Ycal, \Dcal, \xi; \nu) \) be any special test configuration of the polarized spherical log cone \( (Y,D, \xi) \) with angles \( \gamma \) and central fiber \( (Y_0, D_0, \xi) \). Let \( A := A_{(Y_0, D_0)} \) be the log discrepancy of the central fiber. The \textit{Futaki invariant} of \( (\Ycal, \Dcal, \xi; \nu) \) is defined as
\[ \Fut_{\xi}(\Ycal, \Dcal, \nu) := \frac{d _{\xi} \wh{\vol}_{(Y_0,D_0)} (\nu)} {n A_{(Y,D)} (\xi)^{n-1} \vol_{Y_0}(\xi)} = \sprod{\varpi_{\gamma}, \nu} + \frac{\sprod{\varpi_{\gamma}, \xi}} {n} \frac{d_{\xi}\vol_Y(\nu)}{ \vol_Y(\xi)}. \]
\end{defn}

\begin{defn}
We say that a polarized spherical log cone \( (Y, D, \) \(  \xi) \) is \( G \)-equivariantly K-semistable if for every special \( G \)-equivariant test configuration defined by \( \nu \in \Vcal \),  \( \Fut_{\xi}(\Ycal, \Dcal, \nu) \geq 0 \). 

Moreover, \( (Y, D, \xi) \) is \( G\)-equivariantly K-stable (or K-polystable in \cite{LWX})  if it is K-semistable and that \( \Fut_{\xi}(\Ycal, \Dcal, \nu) = 0 \) only if \( (\Ycal, \Dcal, \xi; \nu) \) is a trivial test configuration. 
\end{defn}

The following lemma allows to prove the main theorem by reducing to the computation of the Futaki invariant of a horospherical cone. 

\begin{lem} \label{lemma_futaki_invariant_constant_along_fibers}
Let \( (\Ycal, \Dcal, \xi; \nu) \) be a degeneration with horospherical central fiber \( (Y_0,D_0) \). The Futaki invariant of \( (Y, D, \xi) \) is the same as the Futaki invariant of \( (Y_0, D_0, \xi) \). 
\end{lem}

\begin{proof}
Since the Futaki invariant as defined by Collins-Székelyhidi only depends on the moment cone of \( Y \) (that is the convex cone generated by the weights of \( T_H \)), and that the central fiber \( Y_0 \) has the same moment cone as \( Y \) by a theorem of Knop \cite[Satz 5.4]{Kno90},   the result then follows. 
\end{proof}

Let us now compute the Futaki invariant of a pair associated to a horospherical conical embedding \( G/H \subset Y \). Recall that \( G/H \) is a equivariant torus bundle over \( G/P \), where \( P := N_G(H) \) is the right-stabilizer of the open Borel orbit. Denote by \( \Phi_{P^u} \) the root system of the reductive part \( P^u \).  By Brion's description of the canonical divisor, \( K_Y \) can be represented by 
\[ - K_Y = \sum_{\nu \in \Vcal_Y} D_{\nu} + \sum_{d \in \Dcal_Y}  a_d d. \]
where \( \Vcal_Y \) is the set of \( G\)-stable divisors of \( Y \) and \( \Dcal_Y \) the set of colors of \( Y \), and \( a_d \) are coefficients that depend only on \( G/H \). 

\begin{lem} \label{lemma_futaki_invariant_horospherical}
Let \( (Y, D, \xi) \) be a polarized horospherical log cone with angles \( \gamma \), colored cone \( \Ccal_Y \) and Reeb cone \( \Ccal_R := - \text{int}(\Ccal_Y) \). Let \( \Delta_{\xi} = \set{ \sprod{., \xi} = n} \cap \Ccal_Y^{\vee} \) and \( \text{bar}_{DH}(\Delta_{\xi}) \) be the barycenter of \( \Delta_{\xi} \) with respect to the Duistermaat-Heckman measure
\[ P(p) d \lambda(p) := \prod_{\alpha \in \Phi_{P^u}}  \sprod{\alpha,p} d \lambda(p) \]
For every \( \xi \in \Ccal_R \) and \( \nu \in \Vcal \), the Futaki invariant of \( (Y,D,\xi) \) can be written as
\[ \text{Fut}_{\xi}(Y, D, \nu) =   \sprod{ - \frac{\sprod{\varpi_{\gamma},\xi}}{n} \text{bar}_{DH}(\Delta_{\xi}) + \varpi_{\gamma}, \nu}, \]
where \( \varpi_{\gamma} \) can be interpreted as the \( B\)-weight of the canonical section of the Cartier divisor \( -m (K_Y + D) \). 
\end{lem}

\begin{proof}
Let us first work with an usual cone \( Y \). 
A horospherical cone is \( \Qbb\)-Gorenstein if and only if there exists a linear function \( l \in \Mcal_{\Qbb} \) on \( \Ccal_Y \) such that 
\[ \sprod{l,\nu} = 1, \; \sprod{l, \rho(d)} = a_d. \]
This linear function is exactly the \( B\)-weight \( -\varpi \) of the canonical section of \( K_Y \)
\[ -\varpi = \sum_{\alpha \in \Phi_{P^u }} \alpha. \] 
Moreover, one can show as in \cite{Ngh22b} that the unique \( T_H\)-equivariant holomorphic section \( s \) of the Cartier divisor \( -m K_Y \) satisfies 
\[ \Lcal_{\xi} s = -m \sprod{l,\xi} s. \]
It follows from the description of the log discrepancy in terms of \( s \) that \( A_{Y}(w_{\xi}) = -\sprod{l, \xi} = \sprod{\varpi, \xi} \) for every \( \xi \in \Ccal_R \). The case of \( (Y,D) \) follows by replacing (\( -K_Y\), \( \varpi \)) with (\( -(K_Y +D) \), \(\varpi_{\gamma} \)).

We now compute the volume of \( (Y,\xi) \). By continuity of the volume, it suffices to compute \( \text{vol}_Y(\xi) \) for a quasiregular Reeb vector \( \xi \in (\Ccal_R)_{\Qbb} \). Let \( X := Y// \sprod{\xi} \) be the GIT orbifold quotient of \( Y \). It is naturally a log Fano spherical variety endowed with a Hamiltonian action of the torus \( T_H / \sprod{\xi} \), and the moment polytope for this action after normalizing is exactly \( \Delta_{\xi} \). The Duistermaat-Heckman measure on this polytope coincides with \( P d \lambda \). This measure is moreover independent of the choice of \( \xi \), cf. \cite{Li21}). 

In particular, for a horospherical cone \( Y \) polarized by a quasi-regular Reeb element \( \xi\), 
\[ \vol_Y(\xi) = n! \int_{\Delta_{\xi}} P(p) d \lambda(p). \] 
%Moreover, there exists an orbifold line bundle \( L \) over \( X \) such that
%\[H^0(X,kL) = \bigoplus_{\alpha, \sprod{\alpha, \xi} = k} R_{\alpha}, \quad \Cbb[Y] = \bigoplus_k H^0(X, kL). \]
%Set \( d_j := \dim H^0(X,jL) \). We have by Toën's Riemann-Roch theorem for orbifolds,  
%\[ \lim_{k \to \infty} \frac{d_k}{k^{n-1}}  = a_0 = \vol(X,\xi) =  \int_{\Delta_{\xi}} P(p) d \lambda(p), \]
%where \( \vol(X,\xi) \) is the volume of \( X \) with respect to the volume form \( \omega_{\xi} \). 
%Now let \( u_k := \sum_{j=0}^{k-1} (d_j - d_k) \) and \( v_k := k^{n} \). Remark that \( (v_k) \) is strictly increasing, \( \lim v_k = +\infty \) and
%\[ \lim_{k \to \infty} \frac{u_{k+1}-u_k}{v_{k+1} - v_k} = 0,  \]
%hence by the Stolz-Cesàro theorem,
%\[ \lim_{k \to +\infty} \frac{u_k}{v_k} = \lim_{k \to +\infty} \frac{\sum_{j=0}^{k-1} d_j}{k^{n}} - \frac{d_k}{k^{n-1}} = 0. \] 
%In other words,  
%\[ \vol_Y(\xi) = \lim_{k \to \infty} \frac{\bigoplus_{j=0}^{k-1} \dim H^0(X,jL) }{k^{n}/n!} = \lim_{k \to \infty} \frac{\dim H^0(X,kL)}{k^{n-1}/n!} = n!  a_0.  \]
%It follows that
%\[ \text{vol}_Y(\xi) = n! \int_{\Delta_{\xi}} P(p) d \lambda(p), \]
Using the definition of the Gamma function
\[ \Gamma(n+1) = n! = \int_{s > 0} s^{n} e^{-s} ds, \]
and a Fubini argument, we obtain 
\[ \text{vol}_Y(\xi) = \int_{s > 0} \int_{\sprod{.,\xi} = s} e^{- \sprod{p,\xi}} \sprod{p,\xi}^n P(p) d \lambda(p) ds =  \int_{\Ccal_R^{\vee}} e^{-\sprod{p,\xi}} P(p) d \lambda(p). \]
Finally, a direct computation yields 
\[ d_{\xi}( \log \vol_Y )(\nu) = -\sprod{\text{bar}_{DH}(\Delta_{\xi}), \nu}. \] 
The lemma then follows from the definition of the Futaki invariant in terms of normalized volume.
\end{proof}

%\begin{rmk}
%The measure \( DH_T := \sprod{.,\xi}^n P d \lambda \wedge d \sprod{.,\xi} \) is exactly the Duistermaat-Heckman measure on the cone \( \Ccal_R^{\vee} \). In other words, in the smooth case, if \( \omega_{\xi} \) is a Kähler cone metric compatible with \( \xi\) and \( \mu \) the corresponding symplectic map for the \( T_H \) action, then \( \mu_{*} \omega_{\xi}^n = DH_T \). 
%\end{rmk}

 \section{Proof of Theorem \ref{maintheorem_kstability} and \ref{maintheorem_kstabledegeneration}} \label{section_proof_kstability_maintheorems}

\subsection{Proof of Theorem \ref{maintheorem_kstability}}
\begin{thm} \label{theorem_ricci_flat_equivalent_k_stability}
The following conditions are equivalent.
\begin{itemize}
\item A polarized log Fano cone \( (Y,D,\xi) \) admits log Calabi-Yau cone metric with Reeb vector \( \xi \). 
%\item \( (Y,\xi) \) is Ding-stable. 
\item \( (Y,D,\xi) \) is K-stable. 
\end{itemize}
Moreover, it suffices to test these stability conditions over \( G\)-equivariant special test configurations, where \( G\) is a reductive group acting effectively and holomorphically on \( (Y,D,\xi) \). In particular, a \( G\)-spherical cone \( (Y,D,\xi) \) admits \( K\)-invariant log Calabi-Yau cone metrics iff \( (Y,D,\xi) \) is \( G\)-equivariantly K-stable.  
\end{thm}

\begin{proof}
This was essentially proved in \cite{Li21}, see \cite[Theorem 2.9]{Li21}, also \cite[Theorem 1.7]{HL23}. For the reader's convenience, we provide a sketch of proof. 

Let \( \eta \) be the (weak) contact form associated to \( \xi \). The log Calabi-Yau cone equation on \( (Y,D, \xi) \) can be shown (cf. Equation (78) \cite{Li21}) to be equivalent to an equation of the form 
\[ g(\eta) (d \eta)^{n-1} \wedge \eta = d V_Y^{\xi}, \]
where \( g \) is a positive smooth function on the link \( \set{r^2_{\xi} = 1} \). 
Now let \( \xi_0 := \xi - \xi' \) be any other quasi-regular Reeb vector field, and \( \eta_0 = \eta/ \eta(\xi_0)  \) be the contact form with respect to \( \xi_0 \). The Reeb vector \( \xi_0 \) generates a \( \Cbb^{*}\)-action and we identify the Fano orbifold quotient \( (Y,D)// \sprod{v_{\xi_0}} \) with a log Fano variety \( (X,D_X ) \), where \( D_X \) takes into account the ramified divisor. If \( (Y,D) \) is \( G \times \Cbb^{*} \)-equivariant, then \( (X,D_X) \) is \( G\)-equivariant. Translating the above equation in terms of \( \eta_0 \), \( \xi_0 \), we obtain
\[ g(\eta_0)(d \eta_0)^{n-1} \wedge \eta_0 = dV_Y, \]
which is a \( g\)-soliton equation on the quotient \( (Y,D)// \sprod{v_{\xi_0}} = (X,D_X) \) (cf. Equation (104) \cite{Li21}). In particular, \( (Y,D,\xi) \) admits a weak log Calabi-Yau cone metric if and only if any quasi-regular quotient admits a \( g\)-soliton. 

Let \( \zeta := \xi_0 + t \del_t \), where \( t \del_t \) is the holomorphic vector field generating the \( \Cbb^{*} \)-action. The quotient \( (\Ycal, \Dcal) / \sprod{v_{\zeta}} = ( \Xcal, \Dcal_{\Xcal}, -(K_{\Xcal} + \Dcal_{\Xcal})) \) is a test configuration of \( (X,D_X, -(K_X +D_X) ) \). Here, the Cartier divisor \( -(K_{\Xcal} + \Dcal_{\Xcal}) \) is the multiple of the polarizing orbifold line bundle \( \Lcal \) (viewed as a \(\Qbb\)-Cartier divisor) such that \( \Lcal^{*} \backslash \Xcal \simeq \Ycal \backslash \set{0}. \) 

Conversely, any test configuration of \( (X,D_X) \) induces a test configuration of \( (Y,D) \) (by taking the fiberwise cones over \( X \) with respect to the polarization \(-(K_X + D_X) \)). Moreover, the correspondence sends special test configurations to special test configurations, and \( G\)-equivariant test configurations of \( Y \) to \( G\)-equivariant test configurations of \( X \) (if the action of \( \xi_0 \) is compatible with \( G \)).  

Next, can show that the Ding invariant of \( \Ycal \) is exactly the weighted Ding invariant of any quotient test configuration \( (\Xcal, \Dcal, \Lcal) \). The work of Han-Li \cite{HL23} establishes that \( (X,D_X) \) admits a \( g \)-soliton if and only if it is \( g\)-weighted Ding stable. It follows that \( (Y,D,\xi) \) is Ding-stable iff \( (Y,D,\xi) \) admits a weak log Calabi-Yau cone metric, iff any quasi-regular quotient is \( g\)-weighted Ding-stable. 

Finally, since it is enough to check \( g\)-weighted Ding stability of a quasi-regular quotient over \( G\)-equivariant special test configurations \cite[Theorem 7.3]{HL23}, \cite[Theorem 1.15]{Li21}, the polarized cone \( (Y,D,\xi) \) is Ding-stable iff it is Ding-stable over all \( G\)-equivariant special test configurations for a given \( G \). Finally for a special test configuration, the Ding invariant of the polarized cone \( (Y,D,\xi) \) coincides with the Futaki invariant, and the theorem follows. 
\end{proof}

\begin{thm} \label{theorem_kstability_spherical_cone}
Recall that \( \Sigma := (-\Vcal)^{\vee} \). A  polarized spherical log cone \( (Y,D,\xi) \) with angles \( \gamma \) is \( K \)-stable if and only if 
\[ \text{bar}_{DH}(\Delta_{\xi}) - \varpi_{\gamma} \in \emph{RelInt}(\Sigma). \]
\end{thm}

\begin{proof}
This follows from Theorems \ref{theorem_test_configurations_spherical_cone}, \ref{theorem_ricci_flat_equivalent_k_stability}, and our computation of a horospherical cone's Futaki invariant. For simplicity, we work with a \( \Qbb\)-Gorenstein \( G\)-spherical cone \( Y \). Given any \( G\)-equivariant special test configuration \( (\Ycal, \xi; \nu) \), \( \nu \in (-\Vcal) \) of \( (Y,\xi) \) with central fiber \( Y_0 \), we can construct another test configuration of \( Y_0 \) with horospherical central fiber \( Y_0' \). The Futaki invariant of \( (Y_0',\xi) \) is the same as \( (Y_0,\xi) \) by Lemma \ref{lemma_futaki_invariant_constant_along_fibers}, hence the K-semistability condition is equivalent to 
\[ \frac{\sprod{\varpi,\xi}}{n} \sprod{\text{bar}_{DH}(\Delta_{\xi}), \nu} \geq \sprod{\varpi, \nu}, \; \forall \nu \in (-\Vcal). \]
The fact that a Ricci-flat Kähler cone \( (Y,\xi) \) satisfies \( \sprod{\varpi, \xi} = n \) (cf. Prop. \ref{prop_ricci_flat_reeb_normalized}) simplifies further this condition to 
\[ \text{Fut}_{\xi}(Y_0', \nu) = \sprod{\text{bar}_{DH}(\Delta_{\xi}) - \varpi, \nu} \geq 0, \; \forall \nu \in (-\Vcal). \] 
Recall following fact:
\[ \text{RelInt}(\Sigma) = \set{ \sigma, \sprod{\sigma, \nu} > 0, \forall \nu \in (-\Vcal) \backslash \text{lin}(-\Vcal)}. \] 
The combinatorial condition in the statement
\[ \text{bar}_{DH}(\Delta_{\xi}) - \varpi \in \text{RelInt}(\Sigma) \] 
holds if and only if \( \Fut_{\xi}(Y,\nu) > 0, \forall \nu \in (-\Vcal) \backslash \text{lin}(-\Vcal) \). Under this condition, \( (Y,\xi) \) is clearly K-semistable, and the vanishing of \( \text{Fut}_{\xi}(Y,\nu) \) implies that \( \nu \in \text{lin}(\Vcal) \), hence the test configuration defined by \( \nu \) is a trivial test configuration by Theorem \ref{theorem_test_configurations_spherical_cone}.  Conversely, suppose that \( (Y,\xi) \) is K-stable and \( \text{bar}_{DH}(\Delta_{\xi}) - \varpi \notin \text{RelInt}(\Sigma) \). Then there is \( \nu \notin \text{lin} \Vcal \) such that \( \text{Fut}_{\xi}(Y,\nu) = 0 \), i.e. there is a non-trivial test configuration with vanishing Futaki invariant, a contradiction. The theorem is then proved.  By replacing \( \varpi \) with \( \varpi_{\gamma} \), one obtains directly the K-stability criterion for a log pair. 
\end{proof}

\subsection{Proof of Theorem \ref{maintheorem_kstabledegeneration}}

\begin{prop} \label{prop_stable_degeneration_equals_equivariant_stable_degeneration} 
Let \( (W, \xi) \) be any strictly \( K\)-semistable \( G\)-spherical cone. Then there is a \( G\)-equivariant special degeneration of \((W,\xi) \) with K-stable central fiber. Any other such degeneration has \( G\)-isomorphic central fiber. The analogue holds for a strictly K-semistable \( G\)-spherical log pair \( (W, D, \xi) \).
\end{prop}

\begin{proof}

Let \( F \) be the vanishing locus of \( \Fut_{\xi} \) on \( \Vcal_W \), which is a face of \( \Vcal_W \) containing the linear part  \( \text{lin} \Vcal_W  \). We degenerate \( (W, \xi) \) along a valuation \( \nu \in \text{RelInt}(F) \) (cf. Figure \ref{figure_kstable_degeneration}). The resulting central fiber \( (W', \xi) \) then remains K-semistable (cf. Lemma \ref{lemma_futaki_invariant_constant_along_fibers}) with vanishing locus of \( \Fut_{\xi} \) contained in \( \text{lin} \Vcal_{W'} \), hence  K-stable.
    
Indeed, \( \Vcal_{W'} \) can be identified with
\[ \Vcal_{W'} := \Rbb \nu \oplus \pi( \Vcal_W), \]
where \( \pi \) is the quotient map \( \Ncal_{W, \Rbb} = \Ncal_{W', \Rbb} \to (\Ncal_{W} / \Zbb \nu)_{\Rbb} \)  (cf. Remark \ref{remark_central_fiber_data}). Since \( \nu \in \text{RelInt}(F) \), \( \pi(F) \) is a vector space in \( (\Ncal_W / \Zbb \nu)_{\Rbb} \), and the new Futaki vanishing locus \( \Rbb \nu \oplus \pi(F) \) is contained in \( \text{lin} \Vcal_{W'} \).

Uniqueness of the K-stable degeneration follows from \cite{LWX}: two K-stable central fibers are isomorphic as affine varieties, hence if any one of them is \( G\)-invariant, the other can be endowed with the \( G\)-action through the isomorphism. 
\end{proof}

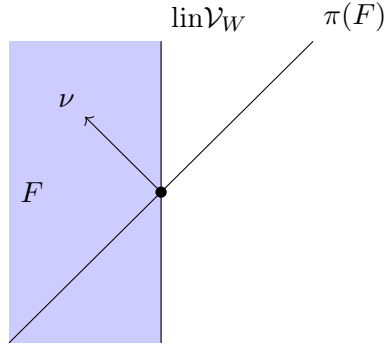
\begin{figure}
\begin{tikzpicture}
\pgfmathsetmacro\ax{2}
\pgfmathsetmacro\ay{0}
\pgfmathsetmacro\bx{2 * cos(120)}
\pgfmathsetmacro\by{2 * sin(120)}
\pgfmathsetmacro\lax{2*\ax/3 + \bx/3}
\pgfmathsetmacro\lay{2*\ay/3 + \by/3}
\pgfmathsetmacro\lbx{\ax/3 + 2*\bx/3}
\pgfmathsetmacro\lby{\ay/3 + 2*\by/3}

\tikzstyle{couleur_pl}=[circle,draw=black!50,fill=blue!20,thick, inner sep = 0pt, minimum size = 2mm]

%couleur
\fill [blue!20] (0,-2)--(0,2)--(-2,2)--(-2,-2)--cycle;
%\fill [black!20] (0,0)--(\ax,\ay)--(\bx,1) -- cycle;
%\node at (\ax, \ay) [couleur_pl] {};
\node at (0,0) [shape=circle, inner sep=1.5pt, fill] {};
%\node at (1,0) [circle, inner sep=1.5pt, fill] {};
%\node at (1,0) [below right] {\( \frac{\alpha^{\vee}}{2} \)};
%\node at (0,1) [circle, inner sep=1.5pt, fill] {};
%\node at (0,1) [below right] {\( \chi \)};
%\node at (-1,0) [circle, inner sep=1.5pt, fill] {};
%\draw[->] (0,0) -- (\ax,\ay) node[below right] {\( \rho(D_{\alpha}) = \alpha^{\vee} \)};
%\draw[->] (0,0) -- (\bx, 1) node[circle, inner sep=1.5pt, fill] {};
\draw[->] (0,0) -- (-1,1) node[above left] { \( \nu \)};
\draw[] (-2,-2) -- (2,2) node[above right]{ \( \pi(F) \)};
%\draw (-2,0)--(0,0);
\draw (0,-2)--(0,2) node[above right]{ \( \text{lin} \Vcal_{W} \)};
\node at (-1.7,0) { \( F \) };
\end{tikzpicture}
\caption{Degeneration of \( (W,\xi) \) along a valuation \( \nu \) in the Futaki vanishing locus \( F\) on \( \Vcal_W \).}
\label{figure_kstable_degeneration}
\end{figure}

\subsection{Examples} \label{section_examples_noncy_cones_with_cone_metrics}

\subsubsection{Horosymmetric cones of rank one} \label{subsubsection_horosymmetric_cones}

%Let \( G/H \) be a symmetric space of rank two with \( \alpha_1, \alpha_2 \) being the short and long roots respectively. Let \( \wh{\alpha}_2, \wh{\alpha}_2' \) be the simple roots of \( G \) restricting to \( \alpha_2 \) in the Satake diagram of \( G/H \). Let \( P \) be the parabolic subgroup of \( G \) containing \( T \) such that \( \wh{\alpha}_2, \wh{\alpha}_2' \) are the only simple roots of \( G \) which are not roots of \( P \). 

Let \( G/H \) be a semisimple horosymmetric space, i.e. an equivariant fibration \( G/H \to G/P \) over a flag manifold with semisimple symmetric fiber \( L/ L \cap H \) (cf. \cite{Del20b} for more details). Supopse that \( G/H \) admits a \( \Qbb\)-Fano embedding. For simplicity, we suppose that \( \text{rank}(G/H) = 1 \), but our arguments extend easily to any rank. 

Let \( \Phi \) be the root system of \( G \) and \( \Phi_L \) be the root system of \( L \) with involution \( \sigma \). Let \( \alpha, 2 \alpha \) be the simple restricted roots with multiplicites \( n_1, n_2 \) induced by \( (\Phi_L, \sigma) \) (where \( n_2 = 0 \) if \( 2 \alpha \) is not a restricted root). Denote by  \( \Phi_s^{+} := \Phi_L^{+} \backslash \Phi^{\sigma} \) and \( \Phi_{Q^u} := \Phi^{+} \backslash \Phi_L \). 

Choosing the horosymmetric subgroup \( H \) such that \( L \cap H = N_G(H) \), we have
\[ \Mcal(G/H) = \Zbb \alpha, \quad \Ncal(G/H) = \Zbb (\alpha^{\vee}/2). \]
Let \( \kappa \) be the Killing form such that \( \sprod{\alpha, \beta^{\vee}} = 2\frac{\kappa(\alpha, \beta)}{\kappa(\beta, \beta)} \). 

Let \( X \) be the \( \Qbb\)-Fano compactification of \( G/H \) (with all the colors) associated to the \( \Qbb\)-reflexive polytope \( Q_X \) \cite{GH15a}. Let \( m > 0 \) be the minimal integer such that \( mK_X \) is Cartier. Take \( Y \) as the Fano cone over \( X \), obtained by contracting the canonical line bundle \( m K_X\) along \( X \).

By construction, \( Y \) is a \( G\times \Cbb^{*}\)-spherical cone with open orbit isomorphic to \( G/H \times \Cbb^{*} \). Here the \( \Cbb^{*} \)-action on \( Y \) comes from the natural \( \Cbb^{*} \)-action on \( mK_X \). For simplicity, we can suppose that \( m = 1\) (so that \( K_X \) is Cartier). 

We endow \( \Mcal(G/H \times \Cbb^{*}) \) with the basis \( (\alpha, \eta) \), where \( \eta \) is the weight of the \( \Cbb^{*}\)-action on \( K_X \). Let \( \Ncal(G/H \times \Cbb^{*}) \) be the dual lattice. The valuation cone of \( G/H \times \Cbb^{*} \) can be identified with the half-space 
\[ \Vcal := \set{(x,y) \in \Ncal(G/H \times \Cbb^{*})_{\Rbb}, \quad x \leq 0 }. \]
and the cone of spherical roots with
\[ \Sigma = (-\Vcal)^{\vee} = \Rbb_{ \geq 0} (\alpha^{\vee}/2). \]
Let \( \varpi \) be the weight of the canonical section of  \( -K_X \), which writes
\[ \varpi :=  \sum_{\wh{\alpha} \in \Phi^{+}_s \cup \Phi_{Q^u} } \wh{\alpha}. \]
The divisor \( -K_X = \sum_{\nu \in \Vcal_X} D_{\nu} + \sum_{d \in \Dcal_X} a_d d \) defines a polytope in \( \Mcal(G/H)_{\Rbb} \)
\[ Q_{X}^{*} := \set{\chi \in \Mcal(G/H)_{\Rbb}, \sprod{\chi, \nu} + 1 \geq 0, \; \sprod{\chi, \rho(d)} \geq -a_d}, \]
which is the dual polytope of \( Q_{X} \) \cite{GH15a}. The moment polytope \( \Delta_X \) of \( -K_X \) is then \( \Delta_X = Q_X^{*} + \varpi \) \cite{GH15a} and we can identify the colored cone of \( Y \) with 
\[ \Ccal_Y = \text{Cone}( Q_X \times \set{1}) = \text{Cone}(Q_X^{*} \times \set{1})^{\vee}, \quad \Dcal_Y = \set{\ol{d \times \Cbb^{*}}, d \in \Dcal_X}. \]
(cf. Figure \ref{figure_colored_cone_horosymmetric_rank_one} for an example). Note that \( \rho( d \times \Cbb^{*}) = (\rho(d), a_d) \), where \( a_d \) is the coefficient of \( d \) in \( -K_X \). The linear function \( (0, 1) \) then defines a linear function on \( \Ccal_Y \) making \( K_Y \) a Gorenstein divisor.

\begin{figure}
\begin{tikzpicture}
\pgfmathsetmacro\ax{2}
\pgfmathsetmacro\ay{0}
\pgfmathsetmacro\bx{2 * cos(120)}
\pgfmathsetmacro\by{2 * sin(120)}
\pgfmathsetmacro\lax{2*\ax/3 + \bx/3}
\pgfmathsetmacro\lay{2*\ay/3 + \by/3}
\pgfmathsetmacro\lbx{\ax/3 + 2*\bx/3}
\pgfmathsetmacro\lby{\ay/3 + 2*\by/3}

\tikzstyle{couleur_pl}=[circle,draw=black!50,fill=blue!20,thick, inner sep = 0pt, minimum size = 2mm]

%couleur
\fill [blue!20] (0,-2)--(0,2)--(-2,2)--(-2,-2)--cycle;
\fill [black!20] (0,0)--(1,1)--(-1,1) -- cycle;
\node at (\ax, \ay) [couleur_pl] {};
\node at (0,0) [shape=circle, inner sep=1.5pt, fill] {};
\node at (1,0) [circle, inner sep=1.5pt, fill] {};
\node at (1,0) [below right] {\( \frac{\alpha^{\vee}}{2} \)};
\node at (0,1) [circle, inner sep=1.5pt, fill] {};
\node at (0,1) [above right] {\( \eta \)};
\node at (-1,0) [circle, inner sep=1.5pt, fill] {};
\draw[->] (0,0) -- (\ax,\ay) node[below right] {\( \rho(d_{\alpha}) = \alpha^{\vee} \)};
%\draw[->] (0,0) -- (\bx, 1) node[circle, inner sep=1.5pt, fill] {};
%\node at (-1,1) [above left] { \( \chi - \frac{\alpha^{\vee}}{2} \)};

\draw (-2,0)--(2,0);
\draw (0,-2)--(0,2);

\end{tikzpicture}
\caption{Colored cone \( (\Ccal_Y, \Dcal_Y) \) of a symmetric rank one conical embedding of \( G/H \times \Cbb^{*} \) where \( G = \SL_2 \), \( H = N_{SL_2}(T) \). Here the lattice of \( G/H \) is generated by the unique restricted root \( \alpha \) (\(= 2 \wh{\alpha} \), \(\wh{\alpha}\) being the unique root of \( \SL_2 \)).  There is a unique color \( d_{\alpha} \) of \( G/H \) with \( \rho(d_{\alpha}) = \alpha^{\vee} \). The polytope \( Q_X \) is then \( \set{ t \alpha^{\vee}2, \abs{t} \leq 1} \), and \( \Ccal_Y^{\vee}\) is the cone over \( Q_X \) in \( \Zbb( \alpha, \eta)_{\Rbb}  \). Note also that \( \varpi = \alpha \) and \( \Delta_X = \set{t \alpha, 0 \leq t \leq 2} \). } 
\label{figure_colored_cone_horosymmetric_rank_one}
\end{figure}
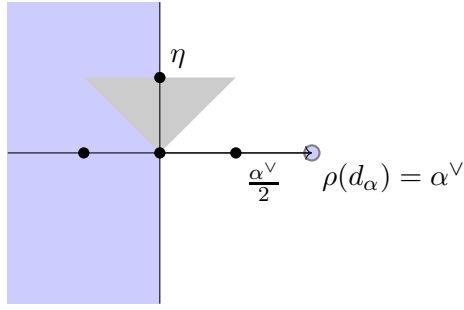

Since the equivariant automorphism group of \( G/H \) is discrete, as \( \Vcal(G/H) \) is only a half-line and \( \dim \Aut_G(G/H) = \dim \text{lin} \Vcal \), the Reeb cone \( \Ccal_R \) of \( Y \) is one-dimensional and can be identified with the positive half-line \( \Rbb_{\geq 0} \eta \). Thus the K-stable Reeb vector of \( Y \), if exists, is unique, so the unique polarization of \( Y \) is given by the polytope \( Q_X \). 

Setting \( 2 \chi := \sum_{\beta \in \Phi_{Q^u}} (\beta + \sigma(\beta)) \),  the Duistermaat-Heckmann polynomial of \( Y \) is defined by 
\[ P_{DH}(p) := p^{n_1+n_2} \prod_{\beta \in \Phi_{Q^u}} \kappa( \beta, 2\chi - p \alpha). \]  

\begin{prop}
The cone \( Y \) is K-stable if and only if
\[ \sprod{\text{bar}_{DH}(Q_X^{*}), \alpha^{\vee}/2} = \sprod{\text{bar}_{DH}(\Delta_X ) - \varpi, \alpha^{\vee}/2 } > 0, \]
i.e. iff \( X \) is K-stable as a \( \Qbb\)-Fano variety. 
\end{prop} 
Note that if \( Y \) is K-stable then any Fano cone over \( X \) obtained by taking a root (or power) of \( mK_X \) and contracting along \( X \) is also K-stable. Repeating the arguments for any rank, we recover in particular the K-stability criterion for \( \Qbb\)-Fano semisimple horosymmetric varieties.

\begin{ex}
Consider the rank one symmetric space \( G = \SL_2 \), \( H = N_{\SL_2}(T) \) and the Fano embedding \( X \) with \( Q_X = \set{ t \alpha^{\vee}/2, \abs{t} \leq 1} \). Then \( P_{DH}(p) = p \) and 
\[ \sprod{\text{bar}_{DH}(Q_X^{*}), \alpha^{\vee}/2} = \frac{\int_{-1}^1 p^2 dp}{ \int_{-1}^1 pdp} = \frac{1}{3} > 0. \] 
\end{ex}

\subsubsection{Horosymmetric cones over boundary divisors of canonical compactifications}

Let us recover state the K-stability result in \cite{BD19} in terms of cone. Consider a rank two semisimple symmetric space \( O \) of rank two, with restricted root system \( R \) generated by long and short simple roots \( \alpha_1, \alpha_2 \) of multiplicities \( m_1, m_2, m_3 \) with \( m_3 \) being the multiplicity of \( 2 \alpha_2 \) which is \( 0 \) if \( 2 \alpha_2 \notin R^{+} \). Let \( P(p) := \prod_{\alpha \in R^{+}} \kappa(\alpha, p) \). 

Let \( D \) be a reduced prime divisor in the boundary of the canonical compactification of a rank two semisimple symmetric space \( O \). The divisor \(D \) is in fact always a rank one horosymmetric variety (but not Fano) \cite{Del20b}.  Consider the Fano blowdown \( D^{\vee} \) of \( D \) along its unique closed orbit with moment polytope \( \Delta \), and take \( \alpha, 2 \alpha \) be the unique restricted positive roots with multiplicities \( n_1, n_2 \). 

\begin{prop} 
Let \( C(D^{\vee}) \) be a Fano cone over \( D^{\vee} \). Then \( C(D^{\vee}) \) has a conical Calabi-Yau metric iff \( \kappa(\text{bar}_{P}(\Delta) - \varpi, \alpha) > 0  \) iff \( D^{\vee} \) is K-stable. 
\end{prop} 

\begin{proof}

The blowdown \( D \to D^{\vee} \) can be seen as the decoloration map, and the colored cone of \( D^{\vee} \) is obtained by adding to the colored cone of \( D \) all the remaining colors. From description of the data of \( D \) and \( \Qbb\)-Fano spherical variety \cite{GH15a} \cite{GH15}, the blowdown \( D^{\vee} \) is then a Fano horosymmetric variety of rank one. 

The combinatorial data of \( D^{\vee} \) can then be deduced from the combinatorial data of the rank two symmetric space \( O \) following \cite[Section 3.2]{BD19}. With the same notation as above, we can take \( \alpha\) to be the restricted root, say, \( \alpha_1 \), and the weight lattice of \( G/H \times \Cbb^{*} \) can be identified with \( \Zbb(\alpha_1,  \eta  = \omega + \lambda_1 \alpha_1 ) \) (cf. \eqref{equation_weyl_intersection}), while the valuation cone of \( G/H \) is a half-line, so the Reeb cone of \( C(D^{\vee}) \) is one-dimensional, hence K-stability of \( D^{\vee} \) is equivalent to that of \( C(D^{\vee}) \). 

The multiplicities of \( \alpha, 2 \alpha \) in \( G/H \) corresponds to their multiplicities as restricted roots in \( O \), namely \( n_1 = m_1, n_2 = 0 \) (\( n_1 = m_2, n_2 = m_3 \) if taking \( \alpha = \alpha_2 \)). 

The anticanonical weight \( \varpi \) of \( D^{\vee} \) then restricts to \( \afrak \) as
\[ 2 \varpi = \sum_{\alpha \in R^{+}} m_{\alpha} \alpha. \] 
Moreover, \( 2 \varpi = (n_1 + 2 n_2) \alpha_1 + 2 \chi \) and 
\[ P_{DH}(p) = P( 2 \varpi - (n_1 + 2 n_2 + p) \alpha_1). \]
The polytope \( \Delta \) is the segment \( \chi + [0, \lambda] \alpha_1 \) where \( \lambda := \lambda_2 - \lambda_1 \) and \( \lambda_{1,2} \) are the intersections of the line \( \varpi + t \alpha_1 \) with the walls of the Weyl chamber
\begin{equation} \label{equation_weyl_intersection}
\lambda_1 := -\frac{\kappa( \varpi, \alpha_2)}{ \kappa(\alpha_1, \alpha_2)}, \quad \lambda_2 := -\frac{\kappa( \varpi, \alpha_1)}{ \kappa(\alpha_1, \alpha_1)}. 
\end{equation}

Remark that \( \kappa(\alpha_1, \chi) = 0 \), hence \( \chi \) is a multiple of the generator of the Weyl chamber. The K-stability criterion of \( C(D^{\vee}) \) can finally be translated in terms of combinatorial data of \( O \) as
\begin{align*} 
\kappa(\text{bar}_{P}(\Delta) - \varpi, \alpha_1)  &= \kappa(\text{bar}_{DH}([0,\lambda]) - (n_1/2 + n_2) \alpha_1, \alpha_1) \\
&=  \frac{\int_{0}^{\lambda} p P_{DH} (p) dp}{ \int_0^{\lambda} P_{DH} (p)dp } - (n_1/2  + n_2) > 0.  
\end{align*} 
\end{proof}

As a corollary, we have 

\begin{prop} \cite[Section 3.3.3]{BD19}
Let  \( \alpha_1, \alpha_2 \) be the long and short root of a rank two symmetric space with restricted root system \( G_2 \) and \( D_1, D_2 \) the divisors in the canonical compactification with restricted root system generated by \( \alpha_1, \alpha_2 \) respectively. The Fano cones \( C(D_1^{\vee}) \), \( C(D_2^{\vee}) \) are respectively K-unstable and K-stable.
\end{prop}

In fact the choices in Section 3.3 of \cite{BD19} should read \( ``\alpha_2 = \alpha, \alpha_2 = \beta \)'' with \( \alpha, \beta \) being their long and short roots.

\section{Valuations and asymptotic cones of Calabi-Yau manifolds} \label{section_semistable_valuations_classification}

\subsection{Donaldson-Sun theory}
Let \( (M, \omega) \) be a \( \del \delb\)-exact complete Calabi-Yau manifold of complex dimension \( n\) with maximal growth and asymptotic cone \( (C,\xi) \), with \( \xi \) being the K-stable Reeb vector. By \cite[Appendix]{DS17}, we also have the Bando-Mabuchi-Matsushima theorem for cones.  

\begin{prop} \cite[Propositions 4.8, 4.9]{DS17}
Let \( G_{\xi} := \Aut_{\xi}(C) \) be the group of holomorphic transformations of \( C \) that preserves \( \xi \). If there exists a Ricci-flat Kähler cone metric on \( C \) with Reeb vector \( \xi \), then \( G_{\xi} \) is reductive, i.e. there is a maximal compact subgroup \( K_{\xi} \) such that 
\[ G_{\xi} = K_{\xi}^{\Cbb}, \] 
and the metric is unique up to the action of the identity component of \( G_{\xi} \). 
\end{prop}

Following \cite{DS17}, the ring of holomorphic functions with polynomial growth \( R(C) \) (with respect to \( \omega_C \)) on \( C \) can be identified with its coordinate ring, and decomposes under the complexified \( T_c \)-action of the Reeb vector as
\[ R(C) = \bigoplus_{\alpha \in \Gamma^{*}} R_{\alpha}, \]
where \( \alpha \) are the \( T := T_c^{\Cbb} \)-action weights. In order to embed \( C \) into \( \Cbb^N \) as an affine subvariety, one can use the local holomorphic embedding \( F_{\infty} \) at the unique fixed point \( O \), and extend it globally to \( C \) using homogeneity under the \( T \)-action.  

\begin{prop} \cite{DS17}
If \( x_1, \dots, x_N \) are local holomorphic functions such that \( F_{\infty} = (x_1, \dots, x_N) \) is the local embedding near \( O \), then the affine cone \( C \) agrees globally with the affine variety generated by \( x_1, \dots, x_N \), i.e. there is a finitely generated ideal \( I_C \) defined by algebraic relations between \( x_1, \dots, x_N \) such that \( C = \Cbb[x_1, \dots, x_N] / I_C \). Under such embedding, the Reeb vector has an extension to \( \Cbb^N \) of the form \( \xi = \Re(i \sum_{a=1}^N w_a z_a \del_{z_a}) \), where \( w_a > 0 \) for all \( a \). 
\end{prop}

For each \( \alpha \in \Gamma^{*} \), the map sending \( \alpha \) to \( \sprod{\alpha, \xi} \) is injective, so we can in fact redecompose \( R(C) \) as
\[ R(C) = \bigoplus_{k} R_{d_k}, \; R_{d_k} := \set{ f_{\alpha_k}, \sprod{\alpha_k, \xi} = d_k}. \]

\begin{defn} The set \( \set{0 = d_0 < d_1 < d_2 < \dots} \) is called the \emph{holomorphic spectrum} of \( C \), denoted by \( \mathcal{S} \). 
\end{defn} 

\begin{prop} \cite[Theorem 3.3]{DS17}
The set \( \mathcal{S} \subset \Rbb_{\geq 0} \) consists of algebraic numbers and is independent of the converging subsequence of \( (M_i, \omega_i) \). In particular, \( \mathcal{S} \) is a finitely generated semigroup. 
\end{prop}

\begin{proof}
The result in \cite{DS17} is stated in the context of \textit{local} tangent cone at a point, but the proof can be adapted almost verbatim for tangent cone at infinity. Fix \( \lambda > 1 \) and let \( (M_i, \omega_i) \) be the rescaling of \( (M,\omega) \) be a factor \( \lambda^{-i} \). Denote by \( \Ccal_{\infty} \) the set of all sequential Gromov-Hausdorff limits of \( (M_i, \omega_i) \) as \( i \to +\infty \). The main ingredients of the proof are the following facts. 
\begin{enumerate}
    \item \( \Ccal_{\infty} \) is compact connected, cf. \cite[Lemma 3.2]{DS17} for a proof which relies on the fact that \( \Kcal(n, \kappa) \) is compact Hausdorff (this is still true for \( \del \delb\)-exact Calabi-Yau metrics).  
    \item From \cite[Lemma 3.5]{DS17}, there is a dense subset \( \Ical \) of \( \Rbb^{+} \) such that if \( D \in \Ical \), then \( N_D := \dim \oplus_{0 < d < D} R_d \) is independent of \( C \in \Ccal_{\infty} \). 
     \item For any \( C \in \Ccal_{\infty} \), we may arrange \( \mathcal{S} \cap (0,D) \) with multiplicities in the increasing order as \( d_1 \leq \dots \leq d_{N_{D}} \), and the map 
     \[ \iota_{D}: \Ccal_{\infty} \to ( \Rbb^{+})^{N_D} \] 
     sending \( C \) to the vector \( (d_1, \dots, d_{N_D}) \) is in fact continuous. 
\end{enumerate}
Since \( \Ccal_{\infty} \) is connected, the image of \( \iota_{D} \) must be a single point for all \( D \in \Ical \), hence \( \mathcal{S} \) is independent of \( C \in \Ccal_{\infty} \).  
\end{proof}

Given a point \( p \in M \), \( \lambda > 0 \), \( B_i := B(p, \lambda^{2i}) \), let \( f \) be a holomorphic function on \( M \), and \( \norm{f}_{i} \) be the \( L^2 \)-norm of \( f|_{B_i} \) with respect to the normalized metric \( \omega_i := \lambda^{-2i} \omega \) restricted to \( B_i \). The \textit{growth rate} of \( f  \) on \( M \) with respect to \(  \omega \) is defined by 
\[ d_{\omega}(f) := \lim_{i \to +\infty} (\log \lambda)^{-1} \frac{\log \norm{f}_{i+1}}{ \log \norm{f}_i}. \] 
\begin{prop} \cite[Corollary 3.8]{DS17}
For every holomorphic function \( f \) on \(  M\), the rate \( d_{\omega}(f) \) is either \( +\infty \) or belongs to \( \mathcal{S} \), and does not depend on the choice of \( p \).  
\end{prop} 

This is stated in the context of local tangent cones, but can be specialized to the case of infinity tangent cones.  We also have the following equivalent characterization:
\[ d_{\omega}(f) = \lim_{r \to \infty} \frac{\sup_{B(p,r)} \log \abs{f}}{\log r}. \]
Hence \( d_{\omega}(f) \) can be seen as the vanishing order at infinity of \( f \), measured with respect to the Calabi-Yau metric \( \omega \). 
Let \( R(M) \) be the ring of holomorphic functions \( f \) with polynomial growth on \( M \), i.e. \( d_{\omega}(f) < +\infty \).

\begin{prop} \cite{DS17} \cite{Liu21}
The ring \( R(M) \) is finitely generated, and 
\[ \wt{M} := \text{Spec}(R(M)) \]
has the structure of an affine variety with isolated singularities. Moreover, there is a map \( \pi_M: M \to \wt{M} \) which is a crepant resolution of singularities. 
\end{prop}

One can easily check that 
\[ \nu_{\omega} := -d_{\omega} \]
extends to a nonpositive (hence never centered) valuation on the quotient field \( \Kcal(M) \) of \( R(M) \), namely
\begin{itemize} 
\item \( \nu_{\omega}(\Cbb^{*}) = 0, \; \nu_{\omega}(0) = +\infty \),
\item \( \nu_{\omega}(fg) = \nu_{\omega}(f) + \nu_{\omega}(g) \), 
\item \( \nu_{\omega}(f+g) \geq \min \set{\nu_{\omega}(f),\nu_{\omega}(g)}. \)
\end{itemize}

\begin{prop} \cite{DS17} 
The possible finite growth rates \( 0 = d_0 < d_1 < \dots \) on \( M \) coincide with \( \mathcal{S} \) and \(\nu_{\omega} \) is a valuation on \( \Kcal(M) \) whose value group \( \nu(\Kcal(M)^{*}) \) is \( \mathcal{S} \cup (-\mathcal{S}) \cup \set{0} \). 
\end{prop}

The degree function \( d_{\omega} \) induces a filtration 
\[ 0 = I_0 \subset I_1 \subset \dots \subset R(M) \]
on \( M \), where \( I_k = \set{f \in R(M), d_{\omega}(f) \leq d_k} \). Moreover, we have \( \dim I_k = \dim \bigoplus_{j \leq k}  R_{d_j} \). 

%\begin{rmk}
%In general, \( \wt{M} \) is in general not isomorphic to \( M \) as a quasiprojective, or even affine variety. The following counterexample in Freudenberg-Moser-Jauslin shows that this already fails in dimension \( 2 \): the surfaces in \( \Cbb^3 \) given by \( x^2 z - y^2  - a \) and \( x^2 y - (1+x)y^2 - a \) are not algebraically isomorphic, but analytically isomorphic. 
%\end{rmk}

Algebraically, \( C \) can be constructed by a \textit{2-step degeneration} as follows. The graded ring
\[ R(W) := \bigoplus I_{k+1}/ I_k \] 
is finitely generated, and can be seen as the central fiber of the filtration induced by the valuation \( \nu_{\omega} \). The affine variety \( W = \text{Spec}(R(W)) \) is the central fiber of a test configuration induced by \( \nu_{\omega} \) with generic fiber isomorphic to \( \wt{M}\). The cone \( W \) is in fact a \textit{weighted tangent cone} at infinity of \( \wt{M} \). 

\begin{prop} \cite{DS17} \cite{SZ22}
Let \( B = B(O,1) \) the unit ball of \( C \) at the fixed point \( O \), embedded in \( \Cbb^N \) using \( F_{\infty}  \), and \( B_i = B(p,2^i) \subset (M, \omega) \) the unit ball on \( (M_i,\omega_i) \). Let \( \Lambda: \Cbb^N \to \Cbb^N \) be linear transformation on \( \Cbb^N \) defined by
\[ \Lambda(z_1, \dots, z_N) = ( (1/ \sqrt{2})^i z_1, \dots, (1/ \sqrt{2})^i z_N), \]
which induces an action on \( F_i \) by 
\[ (\Lambda. F_i) = \Lambda (x_1^i, \dots,x_N^i). \] 
Then there are holomorphic embeddings \( F_i : \wt{M} \to \Cbb^N \) and \( G_i := \Lambda + \tau_i \in G_{\xi} \) for linear maps \( \tau_i \to 0 \), such that 
\begin{itemize} 
\item \( F_{i+1} = G_i \circ F_i \)
\item For any subsequence \( i \to + \infty \), passing to a further subsequence we have \( F_i(\pi_M (B_i)) \to h. F_{\infty}(B) \) in the Hausdorff sense in \( \Cbb^N \) for some \( h \in K_{\xi} \). 
\end{itemize}
Moreover, if  \( M_i := F_i(\wt{M}) \) and \( W_i \) is the weighted tangent cone at infinity of \( M_i\), then \( M_i \simeq M_j \) and \( W_i \simeq W_j \) for all \( i, j \) in the sequence. The elements \( (M_i)_{i \in \Nbb} \) are generic fibers in the special test configuration with central fiber \( W \). 
\end{prop}

We often identify \( (\wt{M}, W) \) with \( (F_1(\wt{M}), W_1) \). Geometrically, \( W \) can be realized by firstly embedding \( \wt{M} \) as an affine variety into \( \Cbb^N \) using holomorphic functions \( F_1 =(x_1, \dots, x_N) \), while diagonally linearizing the \( T_c \)-action on \( \Cbb^N \) with weight \( w = (w_1, \dots, w_N) \in (\Rbb_{> 0})^N \). Define the weight of a monomial \( x_1^{a_1} \dots x_N^{a_N} \) in \( \Cbb^N \) as \( a_1 w_1 + \dots a_N w_N \). Let \(  I\) be the polynomial ideal in \( \Cbb^N\) generating \( \wt{M} \), which is of finite type. For each generator \( f \) of \( I \) (in the Gröbner basis of \( I \) with respect to the ordering induced by \( w \) for example), keep only the term \( f_w \), which consists of monomials  with highest weight. The ideal \( I_w \) generated by all the \( f_w \) then corresponds to \( W \) and \( \Rcal = \Cbb[x_1, \dots, x_N] / I_{w} \). Then \( \Rcal \) admits a natural gradation by \( w \) as
\[ \Rcal = \bigoplus \Fcal_{d_k} / \Fcal_{d_{k+1}}, \]
where \( \Fcal_{d_k} = \set{f \in \Rcal, w(f)  \leq d_k} \). 

\begin{prop} \cite{DS17}
The natural map \( \Rcal \to R(W) \) is an isomorphism and valuation-preserving, namely every element in \( \Fcal_{d_{k+1}} / \Fcal_{d_k} \) is sent to an element in \(I_{k+1} / I_k \). 
\end{prop}

\begin{rmk}
We often identify the weighted valuation \( w \) on \( \Rcal \) with the valuation \( \nu_{\xi} \) on \( R(W) \).   
\end{rmk}

By \cite[Prop 3.26]{DS17}, \( R(W) \) has the same grading as \( R(C) \), hence admits an action of \( T_c \) with the same Hilbert function as \( C \).

\begin{prop} \cite{DS17} 
There is a special test configuration with generic fiber isomorphic to \( W \) and central fiber \( C \). The varieties \( W_i \) are in fact generic fibers in the test configuration.  

Moreover, since \( (C,\xi) \) is K-stable, \( (W,\xi) \) is K-semistable by \cite[Prop. 5.5]{LLX20} and the K-semistable valuation \( \nu_{\xi} \) coincides with the valuation induced by \( \nu_{\omega} \) on \( R(W) \). 
\end{prop} 

%\begin{ex}
%Let \( M \) be the rank one symmetric space \( \SO(n+1) / \SO(n) \) with \( n > 2 \). 
%\end{ex}

\section{Proof of Theorem \ref{maintheorem_kstable_valuation_sphericalmanifold}} \label{section_proof_kstable_valuations_maintheorem}
%Since the properties that we need to prove depend only on the affine crepant blowdown \( \wt{M} \) of \( M \), throughout this section, we work only on smooth affine \( G\)-spherical varieties \( M \) endowed with the \( \del \delb \)-exact and \textit{compatible} \(  K\)-invariant Calabi-Yau metric \( \omega \), i.e. \( M = \wt{M} = \Spec(R(M)) \). 

Before stating key propositions in this section, we make a brief digression to symplectic aspects of spherical varieties. Let \( (X,\omega) \) be a  Kähler manifold with \( K \) acting by holomorphic isometries. A vector field \( \mathbf{X} \) on \( X\) is said to be \textit{locally hamiltonian} if \( \Lcal_{\mathbf{X}} \omega =  0 \). The set \( \text{Ham}_{\text{loc}}(X) \) of locally hamiltonian vector fields on \( X \) is then naturally a Lie algebra. Every smooth function \( H \) on \( X \) defines a locally hamiltonian vector field \( \mathbf{X}_H \) by \( dH = i_{\mathbf{X_H}} \omega \), and there is also a Lie algebra structure on \( C^{\infty}(X) \), called the Poisson structure. The morphism \( \nu : C^{\infty}(X) \to \text{Ham}_{\text{loc}}(X), H \to \mathbf{X}_H \) is in fact a Lie algebra morphism. 

The action of \( K \) is said to be \textit{Poisson} if there is a Lie algebra morphism \( \lambda : \kfrak \to C^{\infty}(X) \), called a \textit{lifting}, such that the morphism \( \nu \circ \lambda \) is exactly the natural Lie algebra morphism \( \kfrak \to \text{Ham}_{\text{loc}}(X) \). Such a lifting exists iff \( K \) acts trivially on the Albanese variety of \( X\) \cite[Proposition 1]{HW}. In particular, on a \( G = K^{\Cbb}\)-projective manifold, \( \text{Alb}(X) \) is trivial (since \( b_1(X) = 0 \)), hence the holomorphic-isometric action of \(  K\) is always Poisson. 

A compact connected Kähler manifold \( (X,\omega) \) with a Poisson \( K\)-action  is said to be a \textit{spherical \( K \)-space} if the Lie subalgebra \( C^{\infty}(X)^K \) is an abelian Lie algebra. 

\begin{thm} \cite[Equivalence Theorem]{HW} \label{theorem_equivalence}
A compact connected Kähler manifold \( (X,\omega) \) with a Poisson \( K\)-action is a \( K\)-spherical space iff it is a projective \( G = K^{\Cbb} \)-spherical manifold. The result is moreover independent of the Kähler structure. 
\end{thm}
The following lemma will be useful to us. 
\begin{lem} \cite[Restriction Lemma]{HW} \label{lemma_restriction}
Let \( X \) be a compact Kähler manifold with a Poisson action of a connected compact group \( K \). If \( X \) is a spherical \( K\)-space, then every closed \( K\)-invariant subvariety of \( X \) is also a spherical \(  K\)-space. 
\end{lem}

%Now let \( X \) be \textit{any} connected Kähler manifold of dimension \( n \) and \( K \) a connected group of holomorphic isometries. Let \( L \) be a \( K \)-linearized holomorphic line bundle over \( X \). We say that \( R(L) := \oplus_{k \in \Zbb} H^0(X, L^k) \) is of \textit{maximal rank} if there are meromorphic functions \( f_1, \dots, f_n \) on \( X\) represented as quotients of sections in \( R(L) \) such that \( df_1 \wedge \dots \wedge df_n \not\equiv 0 \).  If \( R(L) \) is \(  K\)-multiplicity free and of maximal rank, then \( X\) is a \( K \)-spherical space \cite[Theorem 8]{HW}. 

Let us now make a brief recall of valuation theory. The reader may consult \cite{ZS60} or the short notes of Stevensson \cite{Ste17} for more information. Let \( \Kcal/ \Cbb \) be a finitely generated field extension (e.g. \( \Kcal \) is the function field of a complex variety). A complex variety \( X \) is said to be a \textit{model} of \( \Kcal \) if \( \Cbb(X) = \Kcal \). 

Recall the following basic notions. 

\begin{defn}
Let \( \nu \) be a valuation on \( \Kcal/ \Cbb \). 
\begin{enumerate}
\item The \emph{valuation ring} \( R_\nu \) of \( \nu \) is defined as \( R_{\nu} := \set{f \in \Kcal, \nu(f) \geq 0} \). This is a local ring with maximal ideal \( \mathfrak{m}_{\nu} = \set{f \in \Kcal, \nu(f) > 0} \).  \item The field \( \kappa_{\nu} := R_{\nu} / \mfrak_{\nu} \) is said to be the \emph{residue field} of \( \nu \).  
\item The abelian subgroup \( \Gamma_{\nu} := \nu(\Kcal^{*}) \subset \Rbb \) is called the \emph{value group} of  \( \nu \). 
\item The \emph{transcendence degree} of \( \nu \) is \( \text{tr.deg}(\nu) := \text{tr.deg}(\kappa_{\nu} / \Cbb) \). 
\item The \emph{rational rank} of \( \nu \) is \( \text{rt.rk}(\nu) := \dim_{\Qbb}(\Gamma_{\nu} \otimes \Qbb) \).
\end{enumerate}
\end{defn} 

\begin{thm}[Zariski-Abhyankar]
If \( \nu\) is a valuation on \( \Kcal/ \Cbb \), then  
\[ \text{tr.deg}(\nu) + \text{rt.rk}(\nu) \leq \text{tr.deg}(\Kcal/ \Cbb). \] 
\end{thm}

\begin{defn}
A valuation \( \nu \) on \( \Kcal/ \Cbb \) is said to be \emph{Abhyankar} if \( \text{tr.deg}(\nu) + \text{rt.rk}(\nu) = \text{tr.deg}(\Kcal/ \Cbb) \). 
\end{defn}

\begin{defn}
Let \( X \) be a model of \( \Kcal/ \Cbb \). If there is a (generally non-closed) point \( x \in X \) and a local inclusion \( \Ocal_{X,x} \subset R_{\nu} \) of local rings, then the valuation \( \nu \) is said to be \emph{centered} on \( X \), and \( x \) is called the \emph{center} of \( \nu \) on \( X \), denoted by \( c_X(\nu) \). 
\end{defn}

By the valuative criterion for separatedness, if the center of \( \nu \) on a model exists then it is unique, and the valuative criterion of properness guarantees the existence of a center on a proper model. We often identify the center \( c_X(\nu) \) of a valuation with its closure \( \ol{c_X(\nu)} \) inside of the model \( X \) on which the center exists.

\begin{defn}
A valuation \( \nu \) on \( \Kcal / \Cbb\) is said to be \emph{quasimonomial} if there exist 
\begin{enumerate}
    \item a smooth model \( X \) of \( \Kcal/ \Cbb \),
    \item a (generally non-closed) point \( x \in X \),
    \item a regular system of parameters \( y = (y_1, \dots, y_d) \) of the local ring \( \Ocal_{X,x} \) at \( x \), such that \( \nu_{1}, \dots, \nu_{d} \) generate \( \nu(\Kcal^{*}) \cup \set{0} = \Gamma_{\nu} \) as an abelian group.  
\end{enumerate}
\end{defn}

One can in fact take \( x \) to be the \textit{center} of the valuation \( \nu \) on some proper model. 

\begin{thm} \cite[Proposition 2.8]{ELS03} \label{theorem_quasimonomial_abhyankar}
The valuation \( \nu \) is quasimonomial if and only if it is Abhyankar, i.e. 
\[ \text{tr.deg}(\nu) + \text{rt.rk}(\nu) = \text{tr.deg}(\Kcal/ \Cbb). \] 
\end{thm}

\begin{prop} \label{prop_kstable_valuation_invariant}
The valuation \( \nu_{\omega} \) induced by the a \( \del \delb\)-exact complete Calabi-Yau metric \( \omega \) on a quasiprojective manifold \( M \) is quasimonomial.  

If \( M \) admits a \( G\)-spherical action, then \( \nu_{\omega} \) is moreover \( G\)-invariant and identifies with \( -\nu_{\xi} \) in the Cartan algebra of \( M \).
\end{prop}

\begin{proof}
By assumption \( \dim R(M) = \dim R(W) \), hence the quasimonomiality of \( \nu_{\omega} \) follows from a theorem due to Olivier Piltant (cf. \cite[Proposition 3.1]{Tei03} for an accessible reference). 

Next, remark that \( \nu_{\omega} \) is \( K\)-invariant. Indeed, since the metric \( \omega \) is \(  K\)-invariant, every \( k \in K \) defines an isometry between \( B(p,r) \) and \( B(kp,r) \) for any base point \( p \in M \), hence for any meromorphic function \( f \) on \( M \), 
\[ d_{\omega}(k.f) = \lim_{r \to +\infty} \frac{\sup_{B(p,r)} \log \abs{f(k^{-1})} }{\log r} = \lim_{r \to + \infty} \frac{\sup_{B(kp,r)} \log \abs{f} }{\log r}, \]
which is exactly \( d_{\omega}(f) \) as the growth rate at infinity does not depend on the given fixed point. It follows that \( \nu_{\omega} \) is a \(  K\)-invariant valuation. 

 Let us now show that \( \nu_{\omega} \) is \( G\)-invariant. The arguments again use \( K \)-spherical space theory. Let \( Z \) be the center of \( \nu_{\omega} \) in a \( G\)-equivariant smooth projective compactification \( \ol{M} \). In particular, \( \ol{M} \) is a spherical \( K\)-space by Equivalence Theorem \ref{theorem_equivalence}. Since \( \nu_{\omega} \) is \( K \)-invariant, \( Z \) is also a \(  K\)-invariant closed subvariety of \( \ol{M} \), hence a \( K\)-spherical space by Restriction Lemma \ref{lemma_restriction}, which is also \( G\)-spherical again by Equivalence Theorem.  

Let \( \nu' \) be any quasimonomial valuation with center \( Z \). The latter means that there is a \( G\)-equivariant proper birational modification \( Y \to \ol{M} \) with normal crossing divisors \( E_1, \dots, E_m \) such that \( \cap_{i=1}^{r \leq m} E_i \) contains the generic point \( o_Z \) of \( Z \) and \( \nu' \) is a monomial valuation on \( Y \) with center \( Z \). 

Let \( y_1, \dots, y_r \in \Ocal_{Y,o_Z}\) be a system of local parameters such that \( E_i = \set{y_i = 0}, 1 \leq i \leq r \) (by a well-known fact, such \( y_j \) can always be chosen since \( E_1, \dots, E_m \) intersect transversally). By definition, there is a \( r \)-uple \( (\alpha_1, \dots, \alpha_r) \in (\Rbb_{\geq 0}^{+})^r  \) satisfying \( \nu' = \sum_{i=1}^r \alpha_i \text{ord}_{E_i} \). Since \( E_i \) is \( G\)-invariant, \( \text{ord}_{E_i} \) is also \( G\)-invariant, hence \( \nu' \) is \( G\)-invariant. Thus every quasimonomial valuation with center \( Z \) is \( G\)-invariant, hence \( \nu_{\omega} \) is \( G \)-invariant.

The fact that the valuation \( \nu_{\omega} \) corresponds to the valuation induced by the Reeb vector \( \xi \) of the K-stable cone \( (C,\xi) \) can be seen as follows. Since the K-semistable Reeb vector of \( W \) is the same as the K-stable Reeb vector of \( C \),  it is enough to show that \( d_{\omega} = -\nu_{\omega} \) corresponds to the K-stable valuation \( \nu_{\xi} \) of \( (W,\xi) \) induced by \( \xi \).

Let \( G/H \subset M \) and \( G/H_0 \) be the open \( G\)-orbits in \( M \) and \( W \). Since \( R(M) \) and \( R(W) \) are isomorphic as \( G\)-modules by construction, their weight lattices are the same, i.e. \( \Mcal(G/H) = \Mcal(G/H_0) =: \Mcal \). Let \( f_{\infty} \in I_{k+1}^{(\alpha)}/I_k^{(\alpha)} = R(W)^{(\alpha)} \) be any nonzero element and \( f \in I_{k+1}^{(\alpha)} \) a lift. Since \( d_{\omega} \) induces \( \nu_{\xi} \), we have
\[ d_{\omega}(f) = \nu_{\xi}(f_{\infty}). \]
The equality is moreover independent of the choice of \( f \). Finally, from Remark \ref{remark_bmodule_tmodule} it follows that
\[ -\sprod{\alpha, \nu_{\omega}} = d_{\omega}(f) = \nu_{\xi}(f_{\infty}) = \sprod{\alpha_H, \nu_{\xi}} = \sprod{\alpha, \nu_{\xi}}. \]
This terminates our proof. 
\end{proof}

\begin{prop} \label{prop_asymptotic_cone_spherical}
The semistable cone \( W \) in the two-steps degeneration is a \( G\)-spherical cone. In particular, the asymptotic cone of the \( K \)-invariant Calabi-Yau metric \( (M, \omega) \) is a K-stable \( G\)-spherical affine cone \( (C,\xi) \), which is unique up to a \( G\)-equivariant isomorphism preserving \( \xi \). 
\end{prop}

\begin{proof}
Since \( M\) is a \( G\)-spherical manifold and that \( \nu_{\omega} \) is a \( G\)-invariant valuation, it is immediate that \( W \) is a \( G\)-spherical variety. Finally, by Prop. \ref{prop_stable_degeneration_equals_equivariant_stable_degeneration}, there is a unique \( G\)-equivariant degeneration of \( (W,\xi) \) to \( (C,\xi) \), hence \( C \) is \( G\)-spherical.  
\end{proof}

\begin{rmk}
It may be worth mentioning that to prove the uniqueness of the asymptotic cone, one can alternatively use the construction of the \( G\)-equivariant Hilbert scheme in \cite{AB04a} and then readopt the strategy of \cite{DS17}. We explain briefly the main steps. 

\begin{enumerate} 
\item First, since \( W_i \) and \( C \) have the same positive Hilbert function, the action of the torus \( T \) on \( W \) induces a \( T \)-action on \( C \), and  by \cite{AB04a} there is a projective \( G \times T \)-invariant Hilbert scheme \( \mathbf{H} \) parametrizing polarized affine varieties in \( \Cbb^N \) such that for \( i \) large enough, \( W_i \) and \( C \) define points \( [W_i] \) and \( [C] \) in \( \mathbf{H} \). After extracting a subsequence, one can show that \( [W_i] \) converges to \( [C] \) up to a \( K_{\xi} \) action. 
\item There is a small enough neighborhood \( \mathcal{U} \) of \( C \) in \( \Ccal_{\infty} \) such that any \( C' \in \mathcal{U} \) defines an element in \( \mathbf{H} \). The argument uses compactness of \( \mathbf{H} \). 
\item The stabilizer of \( [C] \) in \( \mathbf{H} \) is in fact \( \text{Aut}(C) \), which is reductive by a Matsushima theorem for cones, i.e. there is a maximal compact subgroup such that \( \text{Aut}(C) = K^{\Cbb} \). 
\item We can apply the equivariant slice theorem for \( ([C], K^{\Cbb}) \), and show that \( [C] \) and \( [C'] \) are in the same \( G_{\xi} \) orbit, hence isomorphic as Ricci-flat Kähler cones. We conclude by connectedness of \( \Ccal_{\infty} \). 
\end{enumerate}
\end{rmk}

\begin{rmk} \label{remark_general_obstruction}
A \( K\)-invariant good Calabi-Yau metric on any affine \( G\)-manifold induces in fact a \( G\)-invariant valuation. The arguments can run as follows. Let \( G_{\nu} \subset G \) be the subgroup stabilizing the induced valuation \( \nu \). Then using the definition of \( \nu \), one can show that \( G_{\nu} \) is in fact closed in \( G\) and contains \( K \), hence coincides with \( G\) as a whole. 

Finally, using the Alexeev-Brion Hilbert scheme, one can build a \( G\)-equivariant degeneration of the K-semistable \( G\)-cone \( W \) to the K-stable \( G\)-cone \( C \) and show that it is unique.
\end{rmk}

\section{Examples} \label{section_examples_kstable_valuations}

\subsection{Smooth affine spherical varieties}

As mentionned in the introduction, any smooth affine \( G\)-spherical variety \( M \) is isomorphic to \( G \times^H V \) where \( H \) is a reductive subgroup of \( G \) such that \( G/H \) is (affine) spherical and \( V \) is a \( H \)-module. 

Our examples will deal with two extreme cases. The first is the case \( V = 0 \), i.e. \( M \) is homogeneous, the second is when \( H = G \), or \( M \) is a spherical \( G \)-module. 
For simplicity, we only consider varieties of rank two. The description of K-stable valuations is as follows. 

\begin{prop} \label{proposition_kstable_valuations_ranktwo}
Let \( (M, \omega) \) be a complete \(  K\)-invariant Calabi-Yau smooth affine \( G\)-spherical manifold. Then the valuation \( \nu_{\omega} \) induced by \( \omega \) corresponds to either
\begin{itemize}
    \item the quasi-regular K-semistable Reeb vector of a non-horospherical asymptotic cones if \( \nu_{\omega} \in \del \Vcal \);
    \item the K-stable Reeb vector of the unique horospherical asymptotic cone of \( M \) if \( \nu_{\omega} \in \text{int}(\Vcal) \). 
\end{itemize}
\end{prop}

\begin{proof}
By spherical theory and previous discussions, if \( \nu_{\omega} \in \text{int}(\Vcal) \), then there is  a test configuration defined by \( \nu_{\omega} \) that degenerates \( M \) to a K-semistable horospherical cone \( (W, \nu_{\omega}) \), hence K-stable. By uniqueness of \( G\)-equivariant K-stable degeneration, \( W \) and \( C \) are \( G\)-equivariantly isomorphic. 

If \( \nu_{\omega} \in \del \Vcal \), then the cone \( (W,\nu_{\omega}) \) is K-semistable, and necessarily quasi-regular since its Reeb cone is a half-line. 
\end{proof}

\subsection{K-stable valuations on indecomposable spherical spaces}

The following lemma allows us to simplify the problem of classifying K-stable valuations on affine homogeneous spaces by supposing that the open orbit is indecomposable. 

\begin{lem}
Let \( (M,\omega) \) be the affine spherical homogeneous space \( G_1/H_1 \times \dots \times G_k /H_k \), endowed with complete \( K_1 \times \dots \times K_k \)-invariant \( \del \delb\)-exact Calabi-Yau metric \( \omega \), such that each factor \( G_i / H_i \) is affine indecomposable and admits a complete \( K_i\)-invariant \( \del \delb\)-exact Calabi-Yau metric \( \omega_i \). 

The K-stable valuation \( \nu_{\omega} \) induced by \( \omega \) is then a product of K-stable valuations \( \nu_{\omega_i} \) on the factors. In particular, the asymptotic cone of \( (M,\omega) \) is the product asymptotic cone.   
\end{lem}

\begin{proof}
Let \( \Gamma \) be the weight monoid of \( M \) and \( C \) the asymptotic cone. Since \( C \) is a \( G \)-equivariant degeneration of \( M \), it has the same weight monoid as \( M \), hence the Reeb cone of \( C \) is the interior of \( (\Rbb_{\geq} \Gamma)^{\vee} \). But \( \Gamma \) is the product of the \( \Gamma_is \), hence the Reeb cone of \( C \) is the product of the Reeb cones of all factors' asymptotic cones. 

The Duistermaat-Heckman volume functional \( \vol_{DH} \) then writes as the product of the volume functionals on each factor, and \( -\nu_{\omega} \) can be identified with the unique minimizer, which is clearly the product of the \( -\nu_{\omega_i} \). 
\end{proof}

\begin{prop} \cite[Table 2]{BD19} \cite[Theorem 4.2]{Ngh24} \label{proposition_ranktwo_symmetric_space_valuations}
Let \( \mathcal{W} \) be the restricted Weyl chamber of a rank two symmetric space, and \( \wt{\alpha}_{1,2} \) the primitive generators. 

\begin{itemize}
\item The unique K-stable valuation on decomposable symmetric spaces of rank two is the product of K-stable valuations on each rank one factor. 

\item On indecomposable symmetric spaces of rank two, there are \( 3 \) K-stable valuations on symmetric spaces of restricted root system \( A_2, BC_2/ B_2 \) which correspond to some rational multiple of \( \wt{\alpha}_{1,2} \) and the unique K-stable horospherical valuation. 

The unique K-stable valuation on symmetric spaces of restricted root system \( G_2 \) is the valuation corresponding to a unique generator of the Weyl chamber.
\end{itemize}
\end{prop} 

\begin{proof}
The construction and K-stability of horosymmetric cones was already done in \cite{BD19} (see also part \ref{subsubsection_horosymmetric_cones} for translation in the cone language). For the reader's convenience, we recall here the construction of the horospherical \( G_2 \)-asymptotic cones and the computation of the K-stable Reeb vector in \cite{Ngh22b} \cite{Ngh24}. 

\textit{Construction of the asymptotic cone.} 

Let \( \wh{S} \) be the set of simple roots with respect to a choice of a Borel. The involution \( \theta \) on the symmetric space induces an involution \( \wh{\theta} \) on \( \wh{S} \). Without loss of generality, we work on symmetric spaces \( G/G^{\theta} \), so that \( \Mcal \) is the lattices generated by the restricted fundamental weights. 

Let \( \alpha_1, \alpha_2 \) be the short and long restricted roots and \( \wh{\alpha}_1, \wh{\alpha}_2 \) be the lifts on \( \wh{S} \) of \( \alpha_1, \alpha_2 \) in the same connected component of the Dynkin diagram.

Let \( I := \wh{S} \backslash \set{\wh{\alpha}_1, \wh{\theta}(\wh{\alpha}_1), \wh{\alpha}_2, \wh{\theta}(\wh{\alpha}_2)} \). The open \( (G_2 \times \Cbb^{*}) \)-orbit \( (G_2/ H_0) \times \Cbb^{*} \) of the cone is uniquely determined by \( \Mcal_I = \Mcal \) (=weight lattice of the symmetric space) and \( I \) (cf. Prop. \ref{proposition_horospherical_space} and Remark \ref{remark_central_fiber_data}). Moreover, \( G/H_0 \) is a fibration over \( G/P_I \) where \( P_I = P(\varpi_{\wh{\alpha}_1}) \cap P(\varpi_{\wh{\theta}(\wh{\alpha}_1)} ) \cap P(\varpi_{\wh{\alpha}_2}) \cap P(\varpi_{\wh{\theta}(\wh{\alpha}_2)} )\).

The colors \( \Dcal \) of \( G_2 /H_0 \times \Cbb^{*} \) are in bijection with \( \wh{S} \backslash I \), and two colors of two roots in the same cycle of \( \wh{\theta} \) have the same image in \( \Mcal_I \). Let \(\wh{\alpha}_i^{\vee}, \alpha_i^{\vee}\) be the coroots and restricted coroots, \( i = 1,2 \). 

When \( m = 1 \) (e.g. \( G_2 / \SO_4 \)), since there is no simple root of \( G_2 \) fixed by \( \theta \) (i.e. all nodes in the Satake diagram are white), we have \( \theta(\wh{\alpha}) = -\wh{\alpha} \), so \( \wh{\alpha}^{\vee}_i |_{\Mcal} = 2 \alpha_i^{\vee} \). 

When \( m = 2 \) (for example \( G_2 \times G_2 / G_2 \)), \( \theta(\wh{\alpha}_i) = - \wh{\theta}(\wh{\alpha}_i) \), hence \( \theta(\wh{\alpha}_i)(\wh{\alpha}_i) = 0 \), so \( \wh{\alpha}^{\vee}_i|_{\Mcal} = \alpha_i^{\vee} \).

It follows that
\[ \rho(\Dcal) = \set{ \wh{\alpha}^{\vee}_1|_{\Mcal}, \wh{\alpha}^{\vee}_2|_{\Mcal}} = 
\begin{cases} 
\set{2 \alpha_1^{\vee}, 2 \alpha_2^{\vee}}, m = 1 \\
\set{\alpha_1^{\vee}, \alpha_2^{\vee}}, m = 2.
\end{cases}
\] 
In both cases, the colored cone of \( C \) is
\( ( \Rbb_{\geq 0} \rho(\Dcal), \Dcal) \).

\begin{figure}
\begin{tikzpicture}
\pgfmathsetmacro\ax{1}
\pgfmathsetmacro\ay{0}
\pgfmathsetmacro\bx{(-3)*sin(150)}
\pgfmathsetmacro\by{-cos(150)}
%\pgfmathsetmacro\lax{2*\ax/3 + \bx/3}
%\pgfmathsetmacro\lay{2*\ay/3 + \by/3}
%\pgfmathsetmacro\lbx{\ax/3 + 2*\bx/3}
%\pgfmathsetmacro\lby{\ay/3 + 2*\by/3}

\tikzstyle{couleur_pl}=[circle,draw=black!50,fill=blue!20,thick, inner sep = 0pt, minimum size = 2mm]

%couleur

\draw[->, thick] (0,0) -- (\ax,\ay) node[below right] {\( {\alpha}_1 \)};
\draw[->, thick] (0,0) -- (\bx, \by) node[above right] {\( {\alpha}_2 \)};
\draw[->, thick] (0,0) -- (\ax + \bx, \ay+\by);
\draw[->, thick] (0,0) -- (2*\ax +  \bx, 2*\ay + \by);
\draw[->, thick] (0,0) -- (3*\ax + 2 * \bx, 3*\ay + 2*\by);
\draw[->, thick] (0,0) -- (3*\ax + \bx, 3*\ay + \by);
%\draw[->] (0,0) -- (\ax + \bx/2, \ay + \by/2) node[above right] {\( \wt{\alpha}_1 \)};

\draw (0,0)--(2,0);
%\draw[->] (0,0)--(\ax/2 + \bx, \ay/2 + \by) node[above right]{\( \wt{\alpha}_2 \)};
%\node at ( 1.7*\lax - 1.7*\lbx, 1.7*\lay - 1.7*\lby) [couleur_pl] {};
%\draw (2*\lax - 2*\lbx, 2*\lay - 2*\lby)--(2*\lbx - 2*\lax, 2*\lby - 2*\lay);
\end{tikzpicture}
\caption{Restricted root system of \( G_2\) symmetric spaces.}
\end{figure}
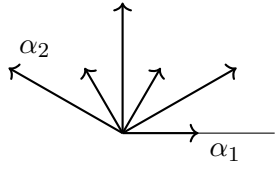

\textit{Reeb vector computation.} 

Recall that \( \kappa(\alpha_1, \alpha_1) = 1 \), \( \kappa(\alpha_2, \alpha_2) = 3 \), and both roots have the same multiplicity \( m \in \set{1,2} \). The positive roots of \( G_2 \) are
\[ \alpha_1, \alpha_2, \alpha_1 + \alpha_2, \alpha_2 + 2 \alpha_1, 2 \alpha_2 + 3 \alpha_1, \alpha_2 + 3 \alpha_1 \]
The half sum of the positive restricted roots (in the Cartan space) is just \( \varpi = 10m \alpha_1 + 6 m \alpha_2 \). 
Recall the setup in \cite{Ngh22b} to compute the Reeb vector \( \xi \). Set \( \delta = \alpha_2 - t \alpha_1, t \in \Rbb \) to be the vector orthogonal to \( \xi \) under \( \kappa \). Identify the valuation cone \( \Vcal \) of the symmetric space with the negative restricted Weyl chamber and the Reeb cone with the positive restricted Weyl chamber \( -\Vcal \). 

Let \( \nu_{\omega} \) be the valuation induced by the \( K \)-invariant Calabi-Yau metric on a \( G_2 \)-symmetric space, then \( \nu_{\omega} \in \Vcal \). By our main Theorem \ref{maintheorem_kstable_valuation_sphericalmanifold}, this is only possible if \( \xi \in -\Vcal \), i.e. iff \( t > 0 \). 

The moment polytope \( \Delta_{\xi} \) can be identified with
\[ \Delta_{\xi} := \set{ \varpi + p \delta, \lambda_{-} \leq p \leq \lambda_{+}}, \quad \lambda_{-} = - \frac{2m}{t+2}, \quad \lambda_{+} = \frac{2m}{2t+3}.\] 
Moreover, the Duistermaat-Heckman polynomial restricted to the Cartan space can be written as
\begin{align*} 
P(p) = &(2m - (2t+3)p)^m (6m + (3t+6)p)^m(8m + (t+3)p)^m \\
&(10m-tp)^m (12m - (3t+3)p)^m (18+3p)^m. 
\end{align*}
Then the Reeb vector is a K-stable polarization iff \( t \) is a solution of
\[ \int_{\lambda_{-}}^{\lambda_{+}} p P(p) dp = 0. \]
For \( m = 1 \) and \( m = 2 \), \( \xi \in - \Vcal \)  iff \( t \) is the \textit{positive} solution of the following respective polynomial equations
\[ 2376 + 9225 t + 13407 t^2 + 9357 t^3 + 3179 t^4 + 424 t^5 = 0,\]
and 
\begin{equation*}
\begin{aligned}
&20558772 + 134444448 t + 374274594 t^2 + 590688162 t^3 + 587394519 t^4 +\\
& 383740299 t^5 + 165293858 t^6 + 45384306 t^7 + 7221048 t^8 + 507988 t^9 = 0. 
\end{aligned}
\end{equation*}
Since all the coefficients are positive, there can be no positive solution. 
\end{proof}

\begin{table}
\adjustbox{width=1\textwidth}{
\centering
\begin{tabular}{|c|c|c|c|c|c|c| }
\hline
Type  & Representative &  \( R \) & Multiplicities & Satake diagram & Hermitian   \\
\hline
\( G \) & \( G_2 / \SO_4 \) & \( G_2 \) & \( 1 \) & \( \dynkin G{oo} \) & \( \text{no} \)  \\

\hline

\( G_2 \) & \( G_2 \times G_2 / G_2 \) & \( - \) & \( 2 \) & \( \begin{dynkinDiagram}[name=upper, edgelength=.75cm] G{o2}
\node (current) at ($(upper root 1)+(0,-.75cm)$) {};
\dynkin[at=(current),name=lower, edgelength=.75cm] G{o2}
\begin{pgfonlayer}{Dynkin behind}
\foreach \i in {1,...,2}
{
\draw[latex-latex]
($(upper root \i)$)
-- ($(lower root \i)$);
}
\end{pgfonlayer}
\end{dynkinDiagram} \) & \( \text{no} \) \\
\hline
\end{tabular}}
\caption{Indecomposable symmetric spaces of restricted root system \( G_2 \). The involution \( \wh{\theta} \) relate two roots connected by the two-sided arrows in the Satake diagram.}
\label{table_rank_two_g2}
\end{table}

As mentionned in the introduction, one can then wonder if there is a Calabi-Yau smoothing of the horospherical \( G_2 \)-asymptotic cone, which would be obtained as the generic fiber of a \( G_2 \)-equivariant deformation of the cone. If this is the case, one can further ask whether a geometric transition phenomenon can occur, that is to prove a crepant resolution of the cone is also Calabi-Yau. The metric would then form a mirror pair with the hypothetical Calabi-Yau smoothing of the cone. This happens for the conifold \( \set{(X,Y,Z,W) \in \Cbb^4, XZ - YW = 0} \) \cite{Gro01}
which is the unique Gorenstein toric cone of dimension \( 3 \) with an isolated \textit{terminal} singularity.

In our case, even if we don't know whether a Calabi-Yau smoothing exists, we can at least affirmatively answer that there can be no \( G_2 \)-equivariant geometric transition.  

\begin{lem} \label{lemma_nocrepantresolution}
There is no equivariant crepant resolution of the horospherical asymptotic cone of \( G_2 \)-symmetric spaces. 
\end{lem}

\begin{proof}
%Let \( \alpha_i^{\vee} \) be the restricted coroots of \( \alpha_i \). Since \( 2 \alpha_i, i = 1,2 \) are not positive roots, it follows that 
%\[ \alpha_1^{\vee} = \frac{2 \alpha_1}{\sprod{\alpha_1, \alpha_1}} = 2 \alpha_1, \quad \alpha_2^{\vee} = \frac{2 \alpha_2}{\sprod{\alpha_2,\alpha_2}} = \frac{2 \alpha_2}{3}. \]
%We also have
%\[ \wh{\alpha}_i^{\vee} = 2 \wh{\alpha}_1, \quad \wh{\alpha}_2^{\vee} = \frac{2 \wh{\alpha}_2}{3}. \] 
%Using the identities \( \alpha_i|_{\afrak} = 2 \wh{\alpha}_i|_{\afrak} \), it follows that
%\[ \alpha_1^{\vee}|_{\afrak} = 2 \wh{\alpha}_1^{\vee}|_{\afrak}, \quad \alpha_2^{\vee} = 2 \wh{\alpha}_2^{\vee}|_{\afrak}. \]
%In particular, the images of the colors of the \( G_2 \)-asymptotic cone inside \( \afrak \) are 
%\[ \sigma(D_1) = \alpha_1^{\vee}/2, \quad \sigma(D_2) = \alpha_2^{\vee}/2. \] 
We use the same notation as in Proposition \ref{proposition_ranktwo_symmetric_space_valuations}.  From \cite{Bri97a} and \cite[Remark 4.3]{GH15}, the anticanonical line bundle of \( C \) can be represented as
\[ -K_{C} = \sum_{\alpha \in \wh{S} \backslash I} a_{\alpha} \ol{D}_{\alpha}, \quad a_{\alpha} = \sprod{\varpi, \alpha^{\vee}}\]
Suppose that \(\pi: X \to C \) is a crepant resolution, then there is a \( G_2 \)-equivariant divisor \( D \subset X \) (corresponding to the primitive vector \( d \) in \( \Mcal \)) such that 
\[ -K_X = \sum_{\alpha \in \wh{S} \backslash I} a_{\alpha} \ol{D}_{\alpha} + D = \pi^{*} (-K_{C_0}) = \sum_{\alpha \in \wh{S} \backslash I} a_{\alpha} \ol{D}_{\alpha} + \frac{2\kappa(\varpi,d)}{ \kappa(d,d)} D, \]
hence \( 2 \kappa(\varpi,d) = \kappa(d,d) \). Let \( d = x \alpha_1 + y \alpha_2 \), with \( x,y \) being positive rationals. Then \( 2 \kappa(\varpi, d) = \kappa(d,d) \) iff
\[ 2m(x-3y)= x(x-3y) + 3y^2 \iff x^2 - x (2m + 3y) + 6my + 3y^2 = 0. \]
It is easy to check by computing the discriminant that for every positive rational \( y \), the equation in \( x \) does not have any solution. 
\end{proof}

%\begin{lem}
%The \( G_2 \)-horospherical asymptotic cones all have orbifold singularities. 
%\end{lem}

\subsection{K-stable valuations on spherical modules}

\begin{figure}
\begin{tikzpicture}
\pgfmathsetmacro\ax{2}
\pgfmathsetmacro\ay{0}
\pgfmathsetmacro\bx{2 * cos(120)}
\pgfmathsetmacro\by{2 * sin(120)}
\pgfmathsetmacro\lax{2*\ax/3 + \bx/3}
\pgfmathsetmacro\lay{2*\ay/3 + \by/3}
\pgfmathsetmacro\lbx{\ax/3 + 2*\bx/3}
\pgfmathsetmacro\lby{\ay/3 + 2*\by/3}

\tikzstyle{couleur_pl}=[circle,draw=black!50,fill=blue!20,thick, inner sep = 0pt, minimum size = 2mm]

%couleur
\fill [blue!20] (0,-2)--(0,2)--(-2,2)--(-2,-2)--cycle;
\fill [black!20] (0,0)--(\ax,\ay)--(\bx,1) -- cycle;
\node at (\ax, \ay) [couleur_pl] {};
\node at (0,0) [shape=circle, inner sep=1.5pt, fill] {};
\node at (1,0) [circle, inner sep=1.5pt, fill] {};
\node at (1,0) [below right] {\( \frac{\alpha^{\vee}}{2} \)};
\node at (0,1) [circle, inner sep=1.5pt, fill] {};
\node at (0,1) [below right] {\( \chi \)};
\node at (0,-1) [circle, inner sep=1.5pt, fill] {};
\node at (0,-1) [below right] {\( -\chi \) };
\draw[->] (0,0) -- (\ax,\ay) node[below right] {\( \rho(D_{\alpha}) = \alpha^{\vee} \)};
\draw[->] (0,0) -- (1,-1) node[circle, inner sep=1.5pt, fill] {};
\node at (1,-1) [below right] { \( -\chi + \frac{\alpha^{\vee}}{2} \)};

\draw (-2,0)--(2,0);
\draw (0,-2)--(0,2);

\end{tikzpicture}
\caption{Colored cone \( (\Ccal, \Dcal) \) of the symmetric manifold \( \Cbb^3 \) with open orbit \( \SO_3/ \SO_2 \times \Cbb^{*} \). The K-stable valuations correspond to the vectors of coordinates \( (1,-1) \) and \( (0,-1) \) in the lattice generated by \( (\alpha^{\vee}/2, \chi) \). }
\label{figure_colored_cone_c3}
\end{figure}
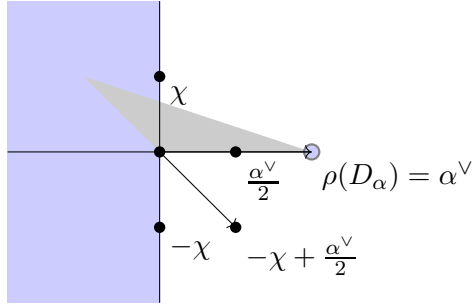

%Any rank two spherical modules has open orbit of the form \( \Cbb^{*} \times G/H \) where \( G/H \) is a rank one spherical space. In particular, \( G/H \) is either a rank one horospherical space or a rank one horosymmetric space. 

Let \( (\rho, V) \) be a regular representation of a connected linear reductive group \( G \) with the induced representation \( (\wh{\rho}, \Cbb[V]) \). Then \( (\rho,V) \) is said to be \textit{multiplicity-free} if the decomposition of \( \Cbb[V] \) into simple \( G\)-modules contains at most one copy of each simple \( G\)-module. A representation \( (\rho, V )\) is multiplicity-free iff \( V \) is a (smooth affine) \( G\)-spherical variety. 

The irreducible multiplicity-free representations were classified by Kac \cite{Kac80} (see also \cite[Theorem 1]{BR96} \cite[Theorem 1.4]{Lea98}).

\begin{thm} \cite{Kac80} \label{theorem_irreducible_sphericalmodule_classification}
The list of multiplicity-free irreducible linear actions of connected reductive linear groups \( G \) is 
\begin{itemize}
    \item[1)] \( \SL_n \), \( \Sp_{2n} \), \( \Lambda^2 \SL_n \) \( (n \; \text{odd}) \), \( \SL_m \otimes \SL_n \) \( (n \neq m \geq 2) \), \( \SL_n \otimes \Sp_4 \) \(( n > 4 )\), \( \text{Spin}_{10} \) when \( G \) is semisimple.  
    %\item[2)] \( S^2 \GL_n \), \( \Lambda^2 \GL_n \), \( \GL_n \otimes \SL_n \), \( \GL_{2,3,4} \otimes \Sp_n \) when \( G \) is non semisimple. 
    \item[2)] \( G \otimes \Cbb^{*} \) with \( G \) being 
    \[ \SL_n, \; \Sp_{2n} (n \geq 2), \;  \SO_n (n \geq 3), \; \text{Spin}_7, \; \text{Spin}_9, \; \text{Spin}_{10}, \; G_2, \; E_6, \]
    and
    \begin{align*} 
    &S^2 \SL_n (n \geq 2), \; \Lambda^2 \SL_n (n \geq 4), \; \SL_m \otimes \SL_n (m, n \geq 2), \\ &\SL_{2,3} \otimes \Sp_{2n} (n \geq 2), \;\SL_n \otimes \Sp_4 (n \geq 4).
    \end{align*}
\end{itemize}
Here:
\begin{itemize} 
\item The index under each group is the dimension of the module.
\item The representation of \( G \) corresponds to \( V(\omega_1) \) where \( \omega_1 \) is the first fundamental weight of \( G \).
\item \( G \otimes G' \) (resp. \( S^2 G \), \( \Lambda^2 G \)) denote the action of \( G \times G' \) on the tensor product \( V(\omega_1) \otimes V(\omega_1') \) (resp. of \( G \) on \( S^2 V(\omega_1) \), \( \Lambda^2 V(\omega_1) \) ). 
\end{itemize}
\end{thm}

The result is extended to the reducible case independently by Benson-Ratcliff \cite{BR96} and A. Leahy \cite{Lea98}. This is done via classification of \textit{indecomposable} spherical modules, namely \( G\)-representations \( (\rho, V) \) that are not equivalent to \( (\rho_1, V_1) \) \( \oplus\) \( \rho_2, V_2) \), where \( (\rho_i, V_i) \) are multiplicity-free representations of \( G_i \) with \( G = G_1 \times G_2 \).  

\begin{lem}
The only non-horospherical multiplicity-free \( G\)-action on a module \( V \) with underlying vector space \( \Cbb^3 \) is given by \( G = \SO_3 \otimes \Cbb^{*} \), where \( \SO_3 \) acts on \( \Cbb^3 \) in the standard way. 
\end{lem}

\begin{proof}
The classification in \cite[Theorem 2]{BR96}, \cite[Theorem 2.5]{Lea98} shows that any indecomposable module must either have one factor (hence belongs to Kac's classification in Theorem \ref{theorem_irreducible_sphericalmodule_classification}), or two factors \( V_i \) each of dimension at least \( 2 \). It follows that any spherical module \( V \) with underlying vector space \( \Cbb^3 \) is indecomposable with only one factor. 

From the list in Theorem \ref{theorem_irreducible_sphericalmodule_classification}, the possible multiplicity-free representations \( (\rho,V) \) with underlying vector space \( \Cbb^3 \) are  \[ (\SL_3, V(\lambda)), (\SL_3 \otimes \Cbb^{*}, V(\lambda)), (\SO_3 \otimes \Cbb^{*}, V(2\omega)), (S^2 \SL_2 \otimes \Cbb^{*}, S^2 V(\omega)),  \]
where \( \lambda  \), \( \omega \) are the fundamental weights of \( \SL_3 \), \( \SL_2 \). 
The first two are horospherical (cf. Prop. \ref{proposition_horospherical_construction}), while the last two are isomorphic via
\[ (S^2 \SL_2, S^2 V(\omega)) \simeq  (\text{PSL}_2, V(2 \omega) \simeq (\SO_3, V(2 \omega)), \]
since \( Z(\SL_2) = \set{\pm 1} \) fixes \( S^2 V (\omega) \simeq V(2\omega) \) and \(\text{PSL}_2 \simeq \SO_3 \). 
\end{proof}

\begin{prop} \label{proposition_calabiyaumetrics_c3}
The K-stable valuations of \( \SO_3(\Rbb) \times \Sbb^1 \)-invariant Calabi-Yau metrics on \( \Cbb^3 \) are
\begin{itemize}
    \item the trivial valuation on the linear part of \( \Vcal \),
    \item the product of the K-stable valuations on the factors \( \SO_3/ \SO_2 \times \Cbb^{*} \). 
\end{itemize}
The former induces a trivial equivariant degeneration, while the latter lies in the interior of \( \Vcal \) and induces a degeneration of \( \Cbb^3\) to the horospherical cone \( A_1 \times \Cbb^{*} \) where \( A_1 \) is the Stenzel asymptotic cone of \( \SO_3/ \SO_2 \) (cf. Example \ref{example_horospherical}). 
\end{prop}

\begin{proof}
Since asymptotic cones are central fibers of equivariant degenerations, one can identify the weight lattice of the cone with the open orbit \( \SO_3/ \SO_3 \times \Cbb^{*} \) of \( \Cbb^3 \) (cf. Remark \ref{remark_central_fiber_data}), which is generated by \( \set{\alpha^{\vee}/2, \chi} \) where \( \alpha \) is the positive (restricted) root of \( \SO_3 \) and \( \chi \) the weight of the \( \Cbb^{*} \)-action on \( \Cbb^3 \) (cf. Figure \ref{figure_colored_cone_c3}). The valuation cone is then
\[ \Vcal = \Rbb_{\leq} \alpha^{\vee} \times \Rbb \chi. \]

From Proposition \ref{proposition_kstable_valuations_ranktwo}, the K-stable valuations of \( \Cbb^3 \) are either in the linear part (with trivial central fiber) or uniquely in \( \text{Int} \Vcal \) (with horospherical cone as central fiber). Since the horospherical central fiber does not depend on the choice of \( \nu \in \text{Int} \Vcal \), it must be \( \SO_3 \times \Cbb^{*}\)-isomorphic to the cone \( A_1 \times \Cbb \). Indeed, an explicit equivariant test configuration can be given by
\[ f = z_1^2 + z_2^2 + z_3^2 : \Cbb^4_{z_0, z_1, z_2, z_3} \to \Cbb, \]
with central fiber \( A_1 \times \Cbb \) \( = \) \( f^{-1}(0) \). Here we view \( \Cbb^4 \) as the spherical module \( \Cbb^3 \times \Cbb \) with an action of \( (\SO_3 \times \Cbb^{*}) \times \Cbb^{*} \), where \( \SO_3 \) acts in the standard way. 

Let \( \omega \) be Li's metric on \( \Cbb^3 \) with corresponding K-stable valuation \( \nu_{\omega} \), asymptotic to \( A_1 \times \Cbb \) (endowed with the horospherical product conical Calabi-Yau metric). From explicit computation in \cite{Ngh24}, the metric on \( A_1 \) has Reeb vector \( \xi = \alpha^{\vee}/2 \). 

The K-stable valuation of the metric on \( A_1 \times \Cbb \) is then \( \nu_{\xi} = ( \alpha^{\vee}/2, \chi) \), hence \( \nu_{\omega} \) corresponds to the vector \( (-\alpha^{\vee}/2, -\chi) \) by Theorem \ref{maintheorem_kstable_valuation_sphericalmanifold}.  
\end{proof}

If we consider any spherical module \( V \) with open orbit of the form \( R_1 \times \Cbb^{*} \) where \( R_1 = G/H \) is any rank one symmetric space, then reasoning as above and using Székelyhidi's uniqueness theorem, one can show that the only Calabi-Yau metrics with the \( G \times \Cbb^{*}\)-symmetry on \( V\) are the standard Calabi-Yau metric and the Li-Conlon-Rochon-Székelyhidi metrics. 

In general, there may exist more of non-horospherical multiplicity-free symmetries of linear reductive groups on \( V \), and one can get a complete list of such actions using \cite{BR96} \cite{Lea98}. However, to get a full classification of metrics with corresponding symmetry, 
the difficulty lies in proving a uniqueness theorem with asymptotic cones \textit{not} of the type \( C \times \Cbb\) with \( C\) having an isolated singularity.

\bibliographystyle{alpha}
\bibliography{biblio}
\end{document}